\documentclass[12pt,reqno]{article}

\usepackage{fullpage}
\usepackage{float}
\usepackage{amsmath,amsthm}

\usepackage{enumerate}
\usepackage{subcaption}

\usepackage[blocks]{authblk}
\usepackage{hyperref}

\theoremstyle{plain}
\newtheorem{theorem}{Theorem}
\newtheorem{corollary}[theorem]{Corollary}

\newtheorem{proposition}[theorem]{Proposition}

\theoremstyle{definition}

\newtheorem{example}[]{Example}

\theoremstyle{remark}

\newcommand{\floor}[1]{\left\lfloor #1 \right\rfloor}
\newcommand{\ceil}[1]{\left\lceil #1 \right\rceil}
\newcommand{\paren}[1]{\left( #1 \right)}

\begin{document}

\title{Card Dealing Math}
\author{Eric Huang}
\author{Timur Kilybayev}
\author{Ryan Li}
\author{Brandon Ni}
\author{Leone Seidel}
\author{Samarth Sharma}
\author{Vivek Varanasi}
\author{Alice Yin}
\author{Boya Yun}
\author{William Zelevinsky}
\affil{PRIMES STEP}
\author{Tanya Khovanova}
\author{Nathan Sheffield}
\affil{MIT}
\date{}
\maketitle

\begin{abstract}
Various card tricks involve under-down dealing, where alternatively one card is placed under the deck and the next card is dealt. We study how the cards need to be prepared in the deck to be dealt in order. The order in which the $N$ cards are prepared defines a permutation.

In this work, we analyze general dealing patterns, considering properties of the resulting permutations. We give recursive formulas for these permutations, their inverses, the final dealt card, and the dealing order of the first card. We discuss some particular examples of dealing patterns and conclude with an analysis of several existing and novel magic card tricks making use of dealing patterns. Our discussions involve 30 existing sequences in the OEIS, and we introduce 44 new sequences to that database.


\end{abstract}

\section{Introduction}

A number of card tricks involve what is known as ``under-down dealing'', or the related down-under dealing, also called ``Australian dealing''. Here, the performer alternates between cycling the top card to the bottom of the deck (placing it ``under'') and dealing the top card of the deck (placing it ``down''). This procedure is typically continued until all cards are dealt.  By understanding the order in which this procedure deals cards, the performer can produce ``magical'' effects. For instance, tricks such as the ``love ritual'' \cite{aragon14} use under-down dealing to ensure that a chosen card always ends up dealt last, despite the fact that the audience has some control over the number of cards in the deck.

The under-down dealing corresponds directly to the classical Josephus problem. In the Josephus problem, people numbered $1, \dots, N$ are arranged in a circle. The procedure starts with person number 1, and, going clockwise, every second person is eliminated until only one person remains. Observe that the order in which people are eliminated is exactly the same as the order in which a sorted deck of cards would be dealt in ``under-down dealing''.

In this paper, we discuss generalizations of under-down dealing to other dealing patterns --- for instance, down-under, under-under-down, etc. Equivalently, we study variants of the Josephus problem where, instead of skipping one person between every elimination, we follow some other pattern of skips and eliminations. For each pattern we consider, we are interested in the total number of moves, the order in which cards must be initially arranged to be dealt in order, the order in which all cards are dealt, the final card to be dealt, and when the top card is dealt. We provide recursive formulas for general patterns, using under-down and down-under dealings as our main examples.

\subsection{Outline}

We start with basic definitions and notation in Section~\ref{sec:prel}. For any pattern, we consider several different associated sequences:
\begin{itemize}
    \item In Section~\ref{sec:nummoves}, we consider the total number of moves (i.e., instances of either placing a card under or down) required to deal $N$ cards.
    \item In Section~\ref{sec:T}, we consider the order of cards such that dealing according to the pattern would produce cards in order. We combine all these patterns into a triangle, and look at the triangle as a sequence.
        \item In Section~\ref{sec:J}, we consider the permutations performed by dealing $N$ cards --- i.e., the order of cards dealt according to the pattern if the cards begin in order. The corresponding triangle is known as the Josephus triangle.
    \item In Section~\ref{sec:freed}, we consider the final card dealt, which is equivalent to the freed person in the Josephus problem.
    \item In Section~\ref{sec:first}, we consider the top card in the prepared deck, which is equivalent to the order of elimination of the first person in the Josephus problem.
\end{itemize}

In each of these sections, we give recursive formulas demonstrating how the sequences and triangles depending on a particular pattern change upon prepending the pattern with either a skip or an elimination. We also discuss how the formulas change after one round of dealing.

We culminate with Section~\ref{sec:sequences}, where we list the sequences that we study. In this paper, we mention 30 existing sequences from the database and discuss 44 new sequences. The new sequences are marked in bold.

In Section~\ref{sec:examples}, we look at examples of particular dealing patterns, such as general periodic patterns, including dealing every $x$th card and patterns of period 3. We also discuss patterns where the number of cards skipped increases by one each time, or depends on the English spellings of numbers or cards. We end the section with more details and explicit formulas related to our main pattern: $UD$.

In Section~\ref{sec:specialdecksizes}, we look at special deck sizes, where for a given pattern, the freed person is the first or the last. This section provides examples of deck sizes useful for magic tricks.

In Section~\ref{sec:magic}, we conclude with a discussion of various existing and new card tricks that make use of under-down and other dealings.

\subsection{Background and related work}

The Josephus problem originates in the first century CE, in the works of Jewish historian Flavius Josephus~\cite{josephus2004great}. In one telling, a group of besieged soldiers decides to die rather than surrender. They stand in a circle and proceed around clockwise, each person killing the person to their immediate left. The story goes that Josephus positioned himself to be the final remaining person after all of the killings and then surrendered to the Romans instead of taking his own life.

Variants of the Josephus problem have appeared since in a number of instances throughout history~\cite{schumer2002josephus}. In mathematical treatments, the focus has been on finding efficient algorithms and non-recursive formulas to determine the $k$th eliminated person when every $x$th person is eliminated \cite{graham1994concrete, lloyd83algorithm, wilf91functional}. There have also been other variants of the problem, including the ``feline'' version, where each person has a fixed number of ``lives'' that must be removed before they are actually eliminated~\cite{ruskey2012feline, sullivan2018variant, ariyibi2019generalizations}.

Some patterns other than executing every $x$th person have also been considered. Beatty and Sullivan first observed the connection to under-down dealing and proposed the ``Texas chainsaw massacre'' version of the Josephus problem, wherein alternatively $x$ people are executed and one is skipped~\cite{sullivan2012structured}. This variant has seen further study and has also been applied to the case where alternatively $x$ people are executed and $y$ are skipped~\cite{sullivan2018variant, park2018serial, ariyibi2019generalizations}.

The appropriate mathematical generalization of Aragon's ``Love Ritual'' routine (popularized by Penn and Teller) was first described by Park and Teixeira~\cite{teixeira2017mathematical}.

\section{Preliminaries}
\label{sec:prel}

\subsection{Dealing patterns}

We consider a deck of $N$ face-down cards, numbered 1 through $N$. Our dealing consists of two types of moves: in an \textit{under} move, the top card is placed at the bottom of the deck, while in a \textit{down} move, the top card is placed on the table face-up. We denote an under move with the letter $U$ and a down move with the letter $D$.

We call an infinite sequence $P$ of letters $U$ and $D$ the \textit{dealing pattern}. When the pattern is periodic, we will often refer to it simply by the periodic portion as shorthand. For instance, the famous ``under-down deal'' has pattern $P = UDUDUDUD\dots$, which we can write in short $P=UD$, and call \textit{$UD$-dealing}. Similarly, the ``down-under deal'', in which one instead begins with a down move, has pattern $P = DUDUDUDU\dots = DU$, which we call in short $DU$-dealing. In this paper, we study a general pattern $P$, which may have a period longer than $2$, or be aperiodic, and use patterns $UD$ and $DU$ as our main examples.

\subsection{Notation}

We denote by $d_i(P)$, correspondingly $u_i(P)$, the index of the $i$th instance of $D$, correspondingly $U$, in pattern $P$. Sometimes, we use notation $d_i$ and $u_i$ to reduce clutter when it is clear what $P$ is. We assume that all our patterns are infinite. Moreover, we assume that $P$ has an infinite number of occurrences of $D$, so that $d_i$ is always defined; otherwise, we cannot deal out large decks.

When we deal the cards in the deck according to a pattern $P$, we will finish at some point. The last deal corresponds to the $N$th letter $D$ in pattern $P$, which has index $d_N$. 

Given a deck of $N$ cards, we call the first $N$ moves a \textit{dealing round}. After a dealing round, we start reusing the cards. The card that is on top is the first card that went under. Suppose we start with pattern $P$, then the pattern for dealing after a round starts with the string where the first $N$ letters are removed from pattern $P$.

For example, suppose our initial pattern is $UD$; then, if we have an even number of cards, after one round, our pattern stays $UD$; otherwise, it becomes $DU$. We can make a similar statement if our starting pattern is $DU$.

We denote by $P_m$ a prefix of $P$ containing $m$ letters. When dealing $N$ cards, we only need to know $P_{d_N}$. We also denote by $|P_m|_D$ and $|P_m|_U$ the number of occurrences of $D$ and $U$ in the prefix $P_m$ correspondingly. In particular, we have $|P_m|_D + |P_m|_U = m$.

We denote by $DP$ (resp. $UP$) the pattern where $D$ (resp. $U$) is prepended in front of $P$. For example, if $P = UD$ and $P' = DU$, then $P = UP'$ and $P' = DP$.

\subsection{Connection to magic}

We started this paper after we learned the following trick. A magician prepares a deck of cards of the same suit, and then performs the $UD$-dealing, revealing that the cards are dealt in order. How is the deck prepared?

We will discuss the formulas later, but for practical purposes, one might use the following procedure.

We lay the cards on the table in order and put them back into the deck in the reverse order of the dealing. We take the string $P_{d_N}$, and starting with the last letter and going backward, we build a deck using the following rules.

\begin{itemize}
\item If the letter is $D$, take the highest-numbered card not in the deck yet and put it face down on top.
\item If the letter is $U$, then move the card on the bottom to the top.
\end{itemize}

\section{Number of moves}
\label{sec:nummoves}

We denote the total number of moves required to deal all of the cards of a given deck size $N$ using pattern $P$ as $M^P(N)$. Note that this always equals the number of letters in the prefix $P_{d_N}$.

\begin{proposition}
    For any pattern $P$, we have $M^P(N)=d_N$.
\end{proposition}

\begin{proof}
    The last deal corresponds to the index of the $N$th occurrence of $D$.
\end{proof}

In several parts of this paper, we consider how properties of a pattern change after prepending the pattern with a single $U$ or a single $D$. The following simple observations show how the values $d_i$, $u_i$, $|P_i|_D$ and $|P_i|_U$ change under these operations.

\subsection{Prepending a pattern \texorpdfstring{$P$}{P} with \texorpdfstring{$D$}{D} or \texorpdfstring{$U$}{U}}

Here is how $|P_i|_D$ and $|P_i|_U$ change upon prepending the pattern with a single letter:
\[|P_i|_D = |DP_{i+1}|_D - 1 = |UP_{i+1}|_D \quad \textrm{ and } \quad |P_i|_U = |DP_{i+1}|_U = |UP_{i+1}|_U - 1.\]
Because of the symmetry between $D$ and $U$, the equations for $u$ are the same as for $d$ with $D$ and $U$ swapped. Similarly, for $d_i$ and $u_i$ we have

\[d_i(P) = d_{i+1}(DP) - 1 = d_i(UP) - 1 \quad \textrm{ and } \quad u_i(P) = u_{i}(DP) - 1 = u_{i+1}(UP) - 1.\]

In particular, this tells us how the number of moves changes when we prepend $D$ or $U$.

\begin{proposition}
We have 
\[M^{DP}(N) = M^{P}(N-1)+1 \quad \textrm{ and } \quad M^{UP}(N) = M^{P}(N)+1.\] 
\end{proposition}

\begin{proof}
We know $M^P(N) = d_N$, so this follows from the previous observation.
\end{proof}

\begin{example}
\label{ex:dN}
We have $M^{UD}(N) = 2N$ and $M^{DU}(N) = 2N-1$.
\end{example}

\subsection{Recursing by one round of dealing} 
We can also write the following recursive formula, this time removing the first $N$ letters as opposed to just the first $1$. Removing $N$ letters corresponds to one round of dealing. Thus, the formulas are especially easy.

\begin{proposition}
    We have
    \[M^{P}(0) = 0 \quad \textrm{ and } \quad M^P(N) = M^{P'}(N - |P_N|_D) + N\]
    where $P'$ is the pattern $P$ with the first $N$ letters removed.
\end{proposition}
\begin{proof}
    By definition, it takes $0$ moves to deal $0$ cards. For $N > 0$, we can first perform one full round of dealing ($N$ moves). This will leave $N - |P_N|_D$ cards remaining, which will take $ M^{P'}(N - |P_N|_D)$ additional moves to deal.
\end{proof}

\section{The dealing triangle}
\label{sec:T}

\subsection{Definition and examples}
A simple trick one can do with dealing patterns is to begin with the cards arranged such that, when dealt according to the pattern, they are dealt in order. Given a dealing pattern $P$, for every $N$, there is some permutation of the numbers $1$ through $N$ such that if the cards begin in that order, they will be dealt in increasing order. We can arrange those permutations in a triangle, letting the $N$th row represent the sequence for a deck of $N$ cards. We denote this triangle $T^P$, and call it the \textit{dealing triangle}. We denote the term in row $N$ and column $k \leq N$ of the dealing triangle $T$ as $T_{N,k}$.

Throughout this paper, we use $UD$ and $DU$ dealing patterns as our primary examples. To reduce clutter, we denote $T^{UD}$ as $T$ and $T^{DU}$ as $T'$. These dealing triangles are shown below in Table~\ref{tab:UDdealing} and Table~\ref{tab:DUdealing}, respectively.
Triangle $T$ in the sequence form is sequence \textbf{A378635}, while triangle $T'$ is sequence \textbf{A378674} in the OEIS \cite{OEIS}.

\begin{table}[ht!]
\begin{minipage}{0.48\textwidth}
\begin{center}
\begin{tabular}{ c c c c c c c c c c}
1  &   &    &   &    &   &    &   &    &   \\
2  & 1 &    &   &    &   &    &   &    &   \\
2  & 1 & 3  &   &    &   &    &   &    &    \\
4  & 1 & 3  & 2 &    &   &    &   &    &   \\
3  & 1 & 5  & 2 & 4  &   &    &   &    &   \\
5  & 1 & 4  & 2 & 6  & 3 &    &   &    &   \\
4  & 1 & 6  & 2 & 5  & 3 & 7  &   &    &   \\
8  & 1 & 5  & 2 & 7  & 3 & 6  & 4 &    &   \\
5  & 1 & 9  & 2 & 6  & 3 & 8  & 4 & 7  &   \\
8  & 1 & 6  & 2 & 10 & 3 & 7  & 4 & 9  & 5 \\
\end{tabular}
\end{center}
\caption{Under-down dealing triangle $T$}
\label{tab:UDdealing}
\end{minipage}
\hfill
\begin{minipage}{0.48\textwidth}
\begin{center}
\begin{tabular}{c c c c c c c c c c c c c c c}
1 \\
1 & 2 \\
1 & 3 & 2 \\
1 & 3 & 2 & 4 \\
1 & 5 & 2 & 4 & 3 \\
1 & 4 & 2 & 6 & 3 & 5 \\
1 & 6 & 2 & 5 & 3 & 7 & 4 \\
1 & 5 & 2 & 7 & 3 & 6 & 4 & 8 \\
1 & 9 & 2 & 6 & 3 & 8 & 4 & 7 & 5 \\
1 & 6 & 2 & 10 & 3 & 7 & 4 & 9 & 5 & 8 \\
\end{tabular}
\end{center}
\caption{Down-under dealing triangle $T'$}
\label{tab:DUdealing}
\end{minipage}
\end{table}

\subsection{Prepending a pattern \texorpdfstring{$P$}{P} with \texorpdfstring{$D$}{D} or \texorpdfstring{$U$}{U}}

We can describe the triangles $T^{DP}$ and $T^{UP}$ in terms of $T^{P}$.

\begin{proposition}
\label{prop:triangleprepend}
We have
\[T_{N,k}^{DP} = \begin{cases}
	1, & \text{ if } k=1; \\
	T_{N-1,k-1}^{P} + 1, & \text{ if } k > 1.
	\end{cases}
\]
We also have
\[T_{N, k}^{UP} = \begin{cases}
        T_{N, N}^{P}, & \text{ if } k=1; \\
        T_{N, k-1}^{P}, & \text{ if } k>1.\\
       \end{cases}
\]
\end{proposition}

\begin{proof}
To perform $DP$-dealing, we will first deal the first card. After that, we will be left to perform $P$-dealing. Our remaining deck has numbers 2 to $N$ in order. This is equivalent to the $P$-dealing with the deck of size $N-1$, where every card number is increased by 1.

In pattern $UP$, we move the top card under before dealing cards according to $P$. After the first move, we now have cards cycled to the left by 1: $\{2, 3, \dots, N, 1\}$. Thus, the corresponding triangle cycles by 1, too.
\end{proof}

In other words, if we remove the first column from triangle $T^{DP}$, then subtract 1 from every element, we get the triangle $T^{P}$. And, if we cycle the triangle $T^{UP}$ to the left, we get the triangle $T^{P}$.

\subsection{Examples for patterns \texorpdfstring{$UD$}{UD} and \texorpdfstring{$DU$}{DU}}
\label{sec:dealing-triangle-examples}

Proposition~\ref{prop:triangleprepend} applied to this case allows us to describe how the triangles $T = T^{UD}$ and $T' = T^{DU}$ are built starting from the first row.

\begin{proposition}
\label{prop:trianglebuilding}
We have $T_{1,1} = T'_{1,1} = 1$. For $N > 1$, we have
\[T_{N,1} = T_{N-1,N-1} + 1 \quad \textrm{ and } \quad T_{N,2} = 1\]
\[T'_{N,1} = 1 \quad \textrm{ and } \quad T'_{N,2} = T'_{N-1,N-1} + 1.\]
In addition, for $k > 2$, we have
\[T_{N,k} = T_{N-1,k-2}+1 \quad \textrm{ and } \quad T'_{N,k} = T'_{N-1,k-2}+1 .\]
\end{proposition}

\begin{proof}
The first row is always 1. Suppose $N > 1$, then, applying Proposition~\ref{prop:triangleprepend} twice, we get
\[T_{N,1} = T'_{N,N} + 1= T_{N-1,N-1} + 1 \quad \textrm{ and } \quad T'_{N,1} = 1\]
and
\[T_{N,2} = T'_{N,1} = 1 \quad \textrm{ and } \quad T'_{N,2} = T_{N-1,1} +1 = T'_{N-1,N-1} + 1.\]
For $k>2$ we get
\[T_{N,k} = T'_{N,k-1} = T_{N-1,k-2} \quad \textrm{ and } \quad T'_{N,k} = T_{N,k-1} + 1 = T'_{N-1,k-2}.\qedhere\]
\end{proof}

\begin{example}
Row 4 in $T$ is 4132; prepending it with 0, we get 04132; moving the last value to the front, we get 20413; adding 1, we get 31524, which is the next row in the same triangle.
\end{example}

The last property of the proposition describes slanted diagonals that we see in both triangles.

Proposition~\ref{prop:trianglebuilding} allows us to express every element in tables $T$ and $T'$ through its values in the first two columns.

\begin{corollary}
\label{cor:tablereduction}
We have 
\[T_{N,2k} = T_{N,2} + k-1 = k \quad \textrm{ and } \quad T_{N,2k+1} = T_{N-k,1} + k\]
and
\[T'_{N,2k} = T'_{N,2} + k-1 \quad \textrm{ and } \quad T'_{N,2k+1} = T'_{N-k,1} + k = k+1.\]
\end{corollary}

That means we can understand the properties of table $T$ if we understand the properties of the first column. As we mentioned before, triangle $T'$ could be expressed through triangle $T$.

\subsection{Recursing by one round of dealing}

When we need the triangle for a particular pattern, we can build it by using the following theorem.

\begin{theorem}
\label{thm:trianglerecursion}
We have
\[T_{N, d_k}^{P} = k \quad \textrm{ and } \quad T_{N, u_k}^{P} = T^{P'}_{|P_N|_U,k} + |P_{N}|_D,
\]
where $P'$ is the pattern $P$ with the first $N$ letters removed.
\end{theorem}

\begin{proof}
Consider the first round of dealing. We are dealing cards face up when the corresponding letter is $D$, which means the card in the $d_k$th place has to be $k$. After that, the number of cards in the deck is reduced by the number of $D$s among the first $N$ letters of $P$, which is $|P_{N}|_D$. We start the new round with the pattern $P'$. Thus, the leftover cards will be distributed in the same way as row $N- |P_{N}|_D$, where we take into account that cards are in the range starting from $|P_{N}|_D + 1$. The rest follows.
\end{proof}

Consider, for example, patterns $UD$ and $DU$. The theorem above allows us to express the corresponding triangles $T$ and $T'$.

\begin{example}
We have
    \begin{align*}
        T_{n,2k} &= k 				& T'_{n,2k+1} &= k+1 \\
        T_{2m,2k+1}  &= T_{m,k+1}  + m 	& T'_{2m,2k} &= T'_{m,k}+m \\
        T_{2m+1,2k+1}  &= T'_{m+1,k+1}+m 	& T'_{2m+1,2k} &= T_{m,k} + (m+1).
    \end{align*}
Suppose we want to calculate  $T_{13, 11}$. There are six $D$s among the first 13 letters of the pattern $UD$. Thus, $T_{13,11} = T'_{7, 6} + 6$, now $T'_{7, 6} = T_{3,3} + (3+1)$, and $T_{3,3} = T'_{2,2} + 1$, while $T'_{2,2}=T'_{1,1} + 1 = 2$, implying that $T_{13,11} = 6+4+1 + 2 = 13$.
\end{example}

\section{The Josephus triangle}
\label{sec:J}

\subsection{The Josephus problem}

Our card-dealing question is related to the Josephus problem. In the Josephus game, $N$ people are standing in a circle waiting to be executed. Counting begins at point 1 in the circle and proceeds around the circle in a specified direction. In the classical variant, after one person is skipped, the next person is executed. The procedure is repeated with the remaining people, starting with the next person, going in the same direction, and skipping one person until only one person remains and is freed.

The classical game corresponds to the $UD$-dealing. We can have the game for any dealing pattern, where for each $U$, we skip a person, and for each $D$, we execute a person.

The game corresponds to our card-dealing trick. Suppose we number people according to row $N$ in triangle $T^P$; then the execution is equivalent to putting down a card. Thus, the person who is executed on $k$th turn corresponds to the card with value $k$. Thus, the row $N$ in table $T^P$ corresponds to numbering people in the circle with their order of execution.

\subsection{Definitions}

In addition to the dealing triangle $T^P$, we can also define another important triangle, which we denote $J^P$ and call \textit{the Josephus triangle}. If we think of each row of $T^P$ as representing a permutation (i.e., the permutation $x \mapsto T_{N,x}^P$), the corresponding row of the Josephus triangle is defined to be the inverse permutation. In terms of the Josephus problem, if we think of the people as numbered $1$ through $N$, the $N$th row lists the people's numbers in the order in which they are executed. So, for instance, if person 4 is eliminated second, the second entry of the row is 4 (i.e.~$J^P_{N,2} = 4$).

As in the previous section, we will take as examples the triangles $J = J^{UD}$ and $J' = J^{DU}$. They are described by sequences A321298 and A378982 in the OEIS and are shown in the Tables~\ref{tab:JUD} and \ref{tab:JDU}. The main diagonals form the sequences $F^{UD}$ and $F^{DU}$, respectively.

\begin{table}[ht!]
\begin{minipage}{0.48\textwidth}
\begin{center}
\begin{tabular}{cccccccc}
1 &   &   &   &   &   &   &   \\
2 & 1 &   &   &   &   &   &   \\
2 & 1 & 3 &   &   &   &   &   \\
2 & 4 & 3 & 1 &   &   &   &   \\
2 & 4 & 1 & 5 & 3 &   &   &   \\
2 & 4 & 6 & 3 & 1 & 5 &   &   \\
2 & 4 & 6 & 1 & 5 & 3 & 7 &   \\
2 & 4 & 6 & 8 & 3 & 7 & 5 & 1 
\end{tabular}
\caption{The Josephus triangle for the $UD$-dealing}
\label{tab:JUD}
\end{center}
\end{minipage}
\hfill
\begin{minipage}{0.48\textwidth}
\begin{center}
\begin{tabular}{c c c c c c c c c c c c c c c}
1 \\
1 & 2 \\
1 & 3 & 2 \\
1 & 3 & 2 & 4 \\
1 & 3 & 5 & 4 & 2 \\
1 & 3 & 5 & 2 & 6 & 4 \\
1 & 3 & 5 & 7 & 4 & 2 & 6 \\
1 & 3 & 5 & 7 & 2 & 6 & 4 & 8 \\
\end{tabular}
\caption{The Josephus triangle for $DU$-pattern}
\label{tab:JDU}
\end{center}
\end{minipage}
\end{table}

\subsection{Prepending a pattern \texorpdfstring{$P$}{P} with \texorpdfstring{$D$}{D} or \texorpdfstring{$U$}{U}}

As in the case of the dealing triangle, the Josephus triangle follows simple recursive relationships upon prepending the dealing pattern with a $U$ or $D$.

\begin{theorem}
We have the following expression of triangle $J^{DP}$ in terms of triangle $J^P$: 
\[ J^{DP}_{N,k}= \begin{cases}
	1, & \text{ if } k=1; \\
	J^{P}_{N-1,k-1} + 1, & \text{ if } k > 1.
\end{cases}\]
For triangle $J^{UP}$, we have \begin{align*}
    J^{UP}_{N, k} = 
    \begin{cases} 
      1, & \text{ if } J^{P}_{N, k} = N;\\
      J^{P}_{N, k} + 1, & \text{ if } J^{P}_{N, k} < N.
    \end{cases}
\end{align*}
\end{theorem}

\begin{proof}
To see the first relationship, observe that once the first card is dealt, we are left with $N-1$ cards numbered $2, \dots, N$. This is equivalent to dealing according to $P$ on cards $1, \dots, N-1$, but with all values increased by $1$.

For the second formula, observe that skipping the first person and dealing according to pattern $P$ is equivalent to dealing with pattern $P$ on cards $2, \dots, N, 1$. Thus, to compute $J^{UP}_{N,k}$ we can compute $J^{P}_{N,k}$ and then add $1 \pmod N$.
\end{proof}

By definition, the last entry of row $N$ in the table $J^P$ is the position of the freed person in the Josephus problem corresponding to pattern $P$. It is also the position of the largest number in each row of $T^P$.

\subsection{Examples for patterns \texorpdfstring{$UD$}{UD} and \texorpdfstring{$DU$}{DU}}

Let us describe triangles $J$ and $J'$ in more detail.

\begin{proposition}
    Triangle $J$ is uniquely defined by the following recursions:
    \begin{align*}
J_{N,k}&=2k & \textrm{ if }  k&\leq\frac{N}{2} \\
J_{2N,k}&=2J_{N,k-N}-1 & \textrm{ if } k&>N \\
J_{2N+1,N+1}&=1 &         \\
J_{2N+1,k}&=2J_{N,k-N-1}+1 & \textrm{ if } k&>N+1.
\end{align*}
Similarly, triangle $J'$ is defined by the following recursions: 
\begin{align*}
J'_{N,k}&=2k-1 & \textrm{ if }  k&\leq \frac{N+1}{2}\\
J'_{2N,k}&=2J'_{N,k-N} & \textrm{ if }  k&>N \\
J'_{2N+1,N+1}&=2N+1  &        \\  
J'_{2N+1,k}&=2(J'_{N,k-N-1}\pmod N)+2 & \textrm{ if }  k&>N+1.\label{eq:josrec:trij':3}
\end{align*}
\end{proposition}

\begin{proof}
The recursions for triangle $J$ are known and provided in comments to sequence A321298 in the OEIS database \cite{OEIS}. The proof is similar to the one for triangle $J'$ below.

The fact that $J'_{N,k}=2k-1$ whenever $k\leq \frac{N+1}{2}$ follows because, in the first round, we deal every odd-numbered card. For $k>\frac{N+1}{2}$, after the first round, the even-numbered cards remain. We consider cases based on the parity of the deck size.

If the total number of cards is even, then we perform $DU$ dealing on the remaining even-numbered cards, so we have $J'_{2N,k}=2J'_{N,k-N}$. If the total number of cards is odd, then we perform $UD$ dealing on the even-numbered cards, which is equivalent to performing $DU$ dealing on cards $4, 6, \dots, 2N, 2$. This gives $J'_{2N+1,k}=2(J'_{N,k-N-1}\mod N)+2$.
\end{proof}

We can express triangles $J$ and $J'$ through the previous row.

\begin{proposition}
We have
    \begin{align*}
    J_{N,k} &= \begin{cases}
        (1 \pmod N) + 1, &\text{if $k = 1$};\\
        (J_{N-1, k-1} + 1 \pmod N) + 1, & \text{otherwise}.
    \end{cases}\\
    J'_{N, k} &=
    \begin{cases}
        1, &\text{if $k = 1$};\\
        (J'_{N-1, k-1}  \pmod {N-1}) + 2, &\text{otherwise}.
    \end{cases}\\
\end{align*}
\end{proposition}

\begin{proof}
    After two steps of $UD$ dealing, we are now performing $UD$ dealing on deck 3, 4, $\dots$, $N$, 1, which gives us the first equation.  After two steps of $DU$ dealing, we are now performing $DU$ dealing on deck 3, 4, $\dots$, $N$, 2, which gives us the second equation.
\end{proof}

Here is an example of how triangles $J$ and $J'$ are connected.

\begin{proposition}
    We can express triangles $J$ and $J'$ through each other.
\begin{align*}
J_{2N,N+k} &= 2J_{N,k} - 1 \\
J_{2N+1,N+k} &= 2J'_{N+1,k} - 1 \\
J'_{2N,N+k} &= 2J'_{N,k} \\
J'_{2N+1,N+1+k} &= 2J_{N,k}.
\end{align*}
\end{proposition}

\begin{proof}
Suppose we want to find $J_{2N,N+k}$. After one round of $UD$ dealing, we are left with a deck containing $N$ cards of increasing odd numbers from 1 to $2N-1$, which we are performing $UD$ dealing on. This gives the first equation.

Now, suppose we want to find $J_{2N+1, N+k}$. After one round of $UD$ dealing, we are again left with $N+1$ cards of increasing odd numbers. But now, the next move is a $D$, so we are performing $DU$ dealing on these cards. This gives the second equation. The remaining two equations are due to exactly the same logic, but starting with $DU$ dealing.
\end{proof}

\begin{example}
The theorem above allows us to construct triangle $J^{DU}$ in Table~\ref{tab:JDU} from triangle $J^{UD}$ from Table~\ref{tab:JUD} above.
\end{example}

\subsection{Recursing by one round of dealing}

As we did with the dealing triangle in Theorem~\ref{thm:trianglerecursion}, we can also consider doing an entire round of dealing before recursing:

\begin{theorem}
\label{thm:josephustrianglerecursion}
We have
\[J_{N, k}^{P} = \begin{cases}
  d_k, & \text{ if } d_k \leq N;\\
  u_{\left(J_{|P_N|_U, k - |P_N|_D}^{P'}\right)}, & \text{ otherwise,}
\end{cases}
\]
where $P'$ is the pattern $P$ with the first $N$ letters removed.
\end{theorem}
\begin{proof}
    In the first round of dealing, our $k$th deal is the index of the $k$th $D$ in our pattern, since we have yet to loop around.
    Then, after the first round of dealing, we have already dealt $|P_N|_D$ cards, and are left to perform $P'$ dealing on a deck of size $|P_N|_U$, whose cards are labeled with all of the indices of $U$s in the first $N$ elements of $P$.
    This yields the stated recursion.
\end{proof}

\subsection{The infinity row}

We see that the columns of the Josephus triangle stabilize at a particular number. Thus, we introduce the notion of the limiting row $J^P_{\infty,k}$, which we call \textit{the infinity row}. The value of $J^P_{\infty,k}$ is the stabilization number for column $k$.

\begin{proposition}\label{prop:stabilization}
For $N \geq M^P(k)$, we have
\[J^P_{N,k} = J^P_{\infty,k} = M^P(k) = d_k(P).\]
\end{proposition}

\begin{proof}
When $k \leq |P_N|_D$, we know that the $k$th deal happens in the first round of dealing. Since we have not yet looped around, the index of the card dealt is the same as the number of moves performed. This means that the $k$th column will stabilize at $M^P(j) = d_k$.
\end{proof}

Column $k$ stabilizes at row $d_k(P)$. We can also observe that the value $J^P_{d_k(P) - 1,k}$ in the column right before stabilization is the same throughout the whole triangle.

\begin{proposition}
For any $P$, and any $k$ with $d_k{P} \geq 2$, we have
    \[J^P_{d_k(P) - 1,k} = u_1(P).\]
\end{proposition}

\begin{proof}
    Suppose $N= d_k(P) -1$. The value $J^P_{N,k}$ corresponds to the first card remaining in the deck after one full dealing round. This is the first card skipped over in the first dealing round, which corresponds to the first $U$ in the pattern. Thus, the value is $u_1(P)$.
\end{proof}

\begin{example}
    If the pattern $P$ starts with $U$, then $u_1 = 1$. We see that this is the value before stabilization in triangle $J^{UD}$. Similarly, we can observe that the value before stabilization in triangle $J^{DU}$ is 2, which equals $u_1(DU)$.
\end{example}

\section{Freed person in the Josephus problem}
\label{sec:freed}

\subsection{Definitions}

In Josephus's problem, we are interested in the person who is freed. This person corresponds to the last card dealt. That means the position of the freed person in the circle equals the position of the largest card in the deck.

We denote the position of the freed person as $F^{P}(N)$. As mentioned, $F^{P}(N)$ is the index of the largest entry of row $N$ in the dealing triangle $T^P$. We also notice that the freed person is the last entry of the $N$th row of the Josephus triangle: $F^{P}(N) = J^P_{N,N}$.

\begin{example}
    If our dealing is $UD$, the position of the freed person is known, and the corresponding sequence $F^{UD}(N)$ is sequence A006257 in the OEIS \cite{OEIS}:
\[1,\ 1,\ 3,\ 1,\ 3,\ 5,\ 7,\ 1,\ 3,\ 5,\ 7,\ 9,\ 11,\ 13,\ 15,\ 1,\ \ldots.\]
 It consists of $n$ blocks, where the $m$th block is a sequence of odd numbers from 1 to $2^m-1$ inclusive. To calculate the $n$th element of this sequence, write $n$ in binary, then rotate left 1 place (this is equivalent to moving the first digit of the number to the last place). For example, 5 in binary is 101, then rotating left, we get 011. Thus, the 5th element of this sequence is 3. Thus, 
\[F^{UD}(N) = 2(N - 2^{\lfloor \log_2 N \rfloor}) + 1.\]
\end{example}

\subsection{Prepending a pattern \texorpdfstring{$P$}{P} with \texorpdfstring{$D$}{D} or \texorpdfstring{$U$}{U}}

Proposition~\ref{prop:triangleprepend} allows us to connect the last person freed for patterns $DP$ and $UD$ to $P$.

\begin{corollary}
The freed person for patterns $DP$ and $UP$ can be calculated as:
\[F^{DP}(1) = 1 \quad \textrm{ and } \quad F^{DP}(N) = F^P(N-1) + 1, \quad \textrm{ for \quad} N > 1,\]
and
\[F^{UP}(N) \equiv 1 + F^{P}(N) \mod{N}.\]
\end{corollary}

\begin{example}
    The index of the freed person for pattern $DU$ is sequence A152423:
\[1,\ 2,\ 2,\ 4,\ 2,\ 4,\ 6,\ 8,\ 2,\ 4,\ 6,\ 8,\ 10,\ 12,\ 14,\ 16,\ 2,\ 4,\ 6,\ 8,\ 10,\ \ldots.\]
We can see that this sequence is A006257 shifted by 1 with an added 1 in front.
\end{example}

\subsection{Recursing by one round of dealing}

The theorem below follows from Theorem~\ref{thm:josephustrianglerecursion} describing the Josephus triangle.

\begin{theorem}
\label{thm:freedpersonrecursion}
We have
\[F^{P}(N) = \begin{cases}
  N, & \text{ if } d_N = N;\\
  u_{F^{P'}(|P_N|_U)}, & \text{ otherwise,}
\end{cases}
\]
where $P'$ is the pattern $P$ with the first $N$ letters removed.
\end{theorem}
\begin{proof}
    This follows by plugging $k = N$ into Theorem~\ref{thm:josephustrianglerecursion}, noting that by definition $F^P(N) = J_{N,N}$.
\end{proof}

Note that the condition $d_N = N$ means that the first $N$ moves of $P$ are all ``down''.

\begin{example}
    Suppose we have 7 cards and perform $UD$ dealing. 
    After one round of dealing, we have removed $|P_7|_D = 3$ cards, and are left to perform $DU$ dealing on the remaining $|P_7|_U = 4$ cards. 
    By Theorem~\ref{thm:freedpersonrecursion}, we have $F^{UD}(7) =u_{F^{DU}(|P_7|_U)} = u_{F^{DU}(4)} = u_4 = 7$.
\end{example}

\section{The elimination order of the first person}
\label{sec:first}

\subsection{Definitions}

Consider the first card in the deck, and suppose its value is $m$ --- that is, suppose it is the $m$th card to be dealt. Correspondingly, the first person in the Josephus problem is executed on step $m$. Thus, the value of the first card in the deck equals the elimination order of the first person in the Josephus problem. We denote the elimination order of the first person as $E^P(N)$.

It follows that the first column in table $T^P$ is the sequence $E^P(N)$. We can also describe $E^P(N)$ as the index of the entry $1$ in $J^P(N)$.

\begin{example}
    The first column in Table~\ref{tab:UDdealing}, which represents $T^{UD}$, is sequence $E^{UD}(N)$, which is sequence A225381 in the OEIS \cite{OEIS}:
\[1,\ 2,\ 2,\ 4,\ 3,\ 5,\ 4,\ 8,\ 5,\ 8,\ 6,\ 11,\ 7,\ 11,\ 8,\ 16,\ 9,\ 14,\ 10,\ 18,\ 11,\ 17,\ 12,\ 23,\ \ldots.\]
\end{example}

\subsection{Prepending a pattern \texorpdfstring{$P$}{P} with \texorpdfstring{$D$}{D} or \texorpdfstring{$U$}{U}}

The order of elimination of the first person is easy to calculate for patterns $DP$ and $UP$.

\begin{proposition}
    We have
    \[E^{DP} = 1 \quad \textrm{ and } \quad E^{UP}(N) = T^P_{N,N}.\]
\end{proposition}

\begin{proof}
If a pattern starts with $D$, then the first person is immediately eliminated. Suppose our pattern is $UP$. From Proposition~\ref{prop:triangleprepend}, the elimination order is $T^{UP}_{N,1} = T^P_{N,N}$. 
\end{proof}

So, $E^{UP}(N)$  is the main diagonal in the triangle $T^P$.

\begin{example}
    We can check that $E^{UD}(N)$  is the main diagonal in the triangle $T' = T^{DU}$.
\end{example}

\subsection{Recursing by one round of dealing}

As $E^P$ is the first column of $T^P$, we can derive recursive formulas for the elimination order of the first person as special cases of our formulas for the dealing triangle.

\begin{theorem}
We have
\[E^P(N) = \begin{cases}
    1, & \text{ if } P_1 = D;\\
    E^{P'}(|P_N|_U) + |P_{N}|_D, & \text {otherwise,}
\end{cases}\]
where $P'$ is the pattern $P$ with the first $N$ letters removed.
\end{theorem}
\begin{proof}
    If $P_1 = D$, then Theorem~\ref{thm:trianglerecursion} gives that $E^P(N) = T^P_{N, 1} = T^P_{N, d_1} = 1$. If $P_1 = U$, then Theorem~\ref{thm:trianglerecursion} gives that $E^P(N) = T^P_{N, 1} = T^P_{N, u_1} = T^{P'}_{|P_N|_U, 1} + |P_N|_D = E^{P'}(|P_N|_U) + |P_{N}|_D$.
\end{proof}

\section{Sequences}
\label{sec:sequences}

For any pattern $P$, we discussed 5 sequences: $M^P(N)$, $T^P$, $J^P$, $F^P(N)$, and $E^P(N)$. We gave examples of patterns $UD$ and $DU$. Here we define several other patterns of interest and summarize the sequences we calculated for those patterns in Tables~\ref{tab:sequences} and~\ref{tab:SLsequences}.

\subsection{Dealing every \texorpdfstring{$x$}{x}th card}

The natural generalization to the most famous pattern $UD$ is the pattern $UUD$. This is equivalent to skipping two people and eliminating the third one in the Josephus problem. We also consider dealing every $x$th card for larger $x$; we discuss details in Section~\ref{sec:xthcard}.

\subsection{Dealings of period 3}

We separately studied all possible dealing patterns of period $3$. The details are in Section~\ref{sec:period3}. Note that pattern $UUU$ is not included as nothing is dealt. The pattern $DDD$ is not included as it is a pattern of period 1.

\subsection{Arithmetic progression}

Instead of skipping the same number of people at each step, we could also consider patterns where the number of skipped people before each elimination varies. For instance, the number of people skipped at each step could be according to an arithmetic progression. We consider the simplest case, where this progression starts with 1 and increases by 1, meaning that we first skip one person and eliminate the next, then skip two people and eliminate the next, and so on. This corresponds to the dealing pattern starting $UDUUDUUUDUUUUD$.

We denote this pattern as $AP$.

\subsection{Spelling numbers}

As they can make fun stories for magic tricks, we further discuss sequences related to English language spellings, explained below.

In this dealing, we spell the number of the next card, putting a card under for each letter in the number, and then we deal. So we start with putting 3 cards under for O-N-E, then deal, then 3 under for T-W-O, then deal, then 5 under for T-H-R-E-E, then deal. The pattern starts as $UUUDUUUDUUUUUD$. We call this dealing the \textit{SpellUnder-Down dealing} and denote it as $SUD$. 

We also consider the \textit{Down-SpellUnder dealing} when we start with dealing a card and then proceed as in $SUD$; we denote this pattern as $DSU$. The pattern $DSU$ is the pattern $SUD$ prepended with $D$. We list sequences related to this pattern in Table~\ref{tab:sequences}, but we do not study these patterns in detail because they depend on English spelling, which is notoriously irregular.

\subsection{Spelling card names}

As this paper is about dealing cards, there is another natural option for a spelling-based pattern: instead of the names of the numbers, we could spell the names of the cards (e.g.\ ``J-A-C-K'' or ``Q-U-E-E-N''). Luckily for us, the English language's quirk gave us the same number of letters for the word ACE as for the word ONE, which means that the spelling-cards dealing triangle in its first 10 rows is the same as $T^{SUD}$, except for replacing 1 with A. Table~\ref{tab:spellcards} shows rows 11, 12, and 13 of the spelling-cards triangle, where A stands for an ace, J for a jack, Q for a queen, and K for a king.

\begin{table}[ht!]
\begin{center}
\begin{tabular}{c c c c c c c c c c c c c}
6 & 7 & 3 & A & J & 5 & 8 & 2 & 10 & 4 & 9 \\
10 & 3 & 5 & A & J & Q & 7 & 2 & 4 & 6 & 8 & 9 \\
3 & 8 & 7 & A & Q & 6 & 4 & 2 & J & K & 10 & 9 & 5 \\
\end{tabular}
\end{center}
\caption{SpellUnder-Down dealing with card names}
\label{tab:spellcards}
\end{table}

The most interesting row is the last. If we prepare the deck of cards of the same suit in this order and then use the SpellUnder-Down dealing with card names, we get all the cards in order:

\begin{center}
3, 8, 7, A, Q, 6, 4, 2, J, K, 10, 9, 5.
\end{center}

Represented as numbers, this is now sequence \textbf{A380248}
\begin{center}
3, 8, 7, 1, 12, 6, 4, 2, 11, 13, 10, 9, 5.
\end{center}

Similarly, for Down-SpellUnder, we get:

\begin{center}
A, J, 4, 6, 2, Q, K, 8, 3, 5, 7, 9, 10.
\end{center}

Represented as numbers, this is now sequence \textbf{A381151}
\begin{center}
1, 11, 4, 6, 2, 12, 13, 8, 3, 5, 7, 9, 10.
\end{center}



\subsection{The table of sequences}

We summarize our data in Table~\ref{tab:sequences}. The new sequences are in bold.

\begin{table}[ht!]
\begin{center}
\begin{tabular}{|c| c | c| c |  c | c |}
\hline
$P$ & $M^P$ 			& $T^P$ 		& $J^P$ 		& $F^P(N)$ 		& $E^P(N)$ \\
\hline
UD  & $2N$: A005843			& \textbf{A378635}	& A321298		& A006257		& A225381 \\
DU  & $2N-1$: A005408$(N) -2$		& \textbf{A378674}	& \textbf{A378682}	& A152423		& A000012 \\
UUD  & $3N$: A008585			& \textbf{A380195}	& \textbf{A381667}	& A054995		& \textbf{A381591} \\
UDU  & $3N-1$:	A016789$(N) - 3$ 	& \textbf{A382354}	& \textbf{A382358}	& \textbf{A382355}	& \textbf{A382356} \\
DUU  & $3N-2$: A016777$(N) - 3$	& \textbf{A381622}	& \textbf{A381623}	& A054995$(N-1)+1$	& A000012 \\
UDD  & 	A007494	& \textbf{A382528}	& \textbf{A381049}	& A337191		& \textbf{A381048} \\
DUD  & 	A032766	& 	\textbf{A381050} 	& \textbf{A383076}	& \textbf{A381051}	& A000012 \\
DDU  & 	A001651	& \textbf{A383847}	& \textbf{A383845}	& \textbf{A383846}	& A000012 \\
UUUD  & $4N$: A008586 & \textbf{A384770}	& \textbf{A384772}		& A088333		& \textbf{A384774} \\
AP  &	A000096 & \textbf{A386639} & \textbf{A386641} & A291317& \textbf{A386643} \\
SUD  &	\textbf{A380202}	& \textbf{A380201}	& \textbf{A380247}	& \textbf{A380204}	& \textbf{A380246} \\
DSU  &	\textbf{A381128} & \textbf{A381127}	& \textbf{A381114}	& \textbf{A381129}	& A000012 \\
\hline
\end{tabular}
\end{center}
\caption{Patterns and corresponding sequences}
\label{tab:sequences}
\end{table}

\subsection{When the first or last person is freed}

Many magic tricks making use of under-down dealing rely on ensuring that the last card dealt is the audience's chosen card. To generalize such tricks to other dealing patterns, we need to know the freed person of those patterns. 

One can design simple tricks that work if the deck size is such that $F^{P}(N)$ is either 1 or $N$. Thus, for a given pattern $P$, we studied the sequences $S^P$ and $L^P$ of deck sizes $N$ such that the last card dealt is card number $1$ or $N$, respectively. In the language of the Josephus problem, we want to find the values of $N$ when the freed person is the first or the last one.

We list the sequence numbers in Table~\ref{tab:SLsequences}. Note that when $P$ starts with $D$, the first person is never freed, so we place N/A in the corresponding row of the table. The new sequences are in bold.

\begin{table}[ht!]
\begin{center}
\begin{tabular}{|c| c | c|}
\hline
$P$ & $S^P$ 			& $L^P$ 		\\
\hline
UD  & A000079	&  A000225 \\
DU  & N/A	&  A000079 \\
UUD  & A081614 & A182459 \\
UDU  & A081615 & A081614\\
DUU  & N/A & A081615\\
UDD  & A038754	& A164123 \\
DUD  & N/A	& A062318 \\
DDU  & N/A	& A038754 \\
UUUD  & \textbf{A385327} & \textbf{A385333}\\
SUD  & \textbf{A385328} & \textbf{A385513} \\
DSU  & N/A &  \textbf{A385328}$(N-1)+1$ \\
AP  & \textbf{A386305} &  \textbf{A386312} \\
\hline
\end{tabular}
\end{center}
\caption{The first and the last person are freed.}
\label{tab:SLsequences}
\end{table}

\section{Particular sequence examples}
\label{sec:examples}

We study several interesting dealing patterns in more detail. For each pattern $P$, we calculate 5 sequences: $M^P(N)$, $T^P$, $F^P(N)$, $E^P(N)$ and $J^P$. The particular sequences are represented in Table~\ref{tab:sequences}. Here, we discuss some general observations.

\subsection{Periodic patterns}
\label{sec:periodicpatterns}

We start with some results that hold for any periodic pattern $P$.

\textbf{Number of moves.} We can calculate the number of moves exactly.

\begin{proposition}
    If $P$ is a periodic pattern of period $p$, then we can write the number of moves as 
    \[M^P(N) = p  \floor{\frac{N}{|P_p|_D}} + M^P(N \pmod p).\]
\end{proposition}
\begin{proof}
    Once we have performed $\floor{\frac{N}{|P_p|_D}}$ full periods of dealing, we will have performed $ p \floor{\frac{N}{|P_p|_D}}$ moves, and dealt $|P_p|_D\cdot \floor{\frac{N}{|P_p|_D}} = N - (N \pmod p)$ cards. So, the number of moves remaining is equal to the number of moves required to deal $N \pmod p$ cards.
\end{proof}

\textbf{The dealing triangle.} As with our main examples of $P=UD$ and $P=DU$ (see Section~\ref{sec:dealing-triangle-examples}), we can describe the dealing triangle recursively for any periodic pattern.

\begin{theorem}\label{thm:periodic-dealing-triangle-later-cols}
    If $P$ is a periodic pattern with period $p$, then for any $k > p$ we have 
    \[T_{N, k} = T_{N-|P_p|_D, k - p} + |P_p|_D.\]
\end{theorem}
\begin{proof}
    After $p$ steps of dealing, we are left to perform $P$ dealing on a deck of size $N-|P_p|_D$. The card originally at index $k > p$ is now the ($k-p$)th card under this ordering. So, we will deal that card after exactly $T_{N-|P_p|_D, k-p}$ more dealings. Since we have already dealt $|P_p|_D$ cards thus far, this means that $T_{N, k} = T_{N-|P_p|_D, k - p} + |P_p|_D$.
\end{proof}

In order to express the entire triangle, we now need to describe the first $p$ columns. For this purpose, we have the following theorem. Note that this theorem only applies beyond the $p$th row, but that the first $p$ rows consist of only a finite number of values, so we find it reasonable to refrain from explicitly computing this ``base case''.

\begin{theorem}\label{thm:periodic-dealing-triangle-first-cols}
    If $P$ is a periodic pattern with period $p$, then for any $N \geq p$ and any $k \leq p$ we have
    \begin{align*}
        T_{N,k} &= |P_k|_D & &\text{ if } P_k = D\\
        T_{N,k} &= T_{N - |P_p|_D, N - p + |P_k|_U} + |P_p|_D & &\text{ if } P_k = U.
    \end{align*}
\end{theorem}
\begin{proof}
    Again, after $p$ steps of dealing, we are left to perform $P$ dealing on a deck of size $N-|P_p|_D$. 
    If $P_k = D$, then the card originally at index $k$ has already been dealt by this point, since it corresponds to the $|P_k|_D$th deal.
    
    Otherwise, the card originally at index $k$ is placed underneath the deck. In the first $p$ steps, the number of cards placed under the deck after the card originally at index $k$ is $|P_p|_U - |P_k|_U$.
    So, after $x$ steps the card originally at index $k$ will be at index $(N - |P_p|_D) - (|P_p|_U - |P_k|_U) = N - p + |P_k|_U$. Since we have already dealt $|P_p|_D$ cards, this means that the card originally at index $k$ will be the $(T_{N - |P_p|_D, N - p + |P_k|_U} + |P_p|_D)$th card dealt. 
\end{proof}

\textbf{The Josephus triangle.}

Here we show how to build the Josephus triangle recursively.

\begin{theorem}\label{thm:josephus-periodic}
If $P$ is a periodic pattern with period $p$, then for any $N \geq p$ and any $k$ we have 
\[J^{P}_{N,k} = \begin{cases}
    d_k, & \text{if $d_k \leq p$ };\\
    J^P_{N-|P_p|_D, k - |P_p|_D} + p, & \text{if $d_k > p$ and $J^P_{N-|P_p|_D, k - |P_p|_D} \leq N-p$};\\
    u_{\paren{p - N + J^P_{N-|P_p|_D, k - |P_p|_D}} }, & \text{otherwise}.
\end{cases}\]

\end{theorem}
\begin{proof}
If $d_k \leq p \leq N$, this means that the $k$th card dealt is dealt in the first round of dealing, meaning that it is the card numbered $d_k$ because we have not yet cycled through the deck.

If $d_k > p$, then the $k$th card is dealt after the first period of dealing concludes. After a single period of dealing, we are left to perform $P$ dealing on the deck of cards $p + 1$, $p+2$, $\dots$, $N$, $u_1$, $u_2$, $\dots$, $u_{|P_p|_U}$. We have already dealt $|P_p|_D$ cards thus far, so this is a deck of $N - |P_p|_D$ cards, and the $k$th card dealt overall will be the $k - |P_p|_D$th card dealt from this deck. Thus, the $k$th card dealt overall is the card at the $J^P_{N-|P_p|_D, k-|P_p|_D}$th index of this new deck. Observing that the card at the $i$th index of this new deck is $p + i$ for $i \leq N - p$, and $u_{p - N -i}$ for $i > N-p$, we get the desired statement.
\end{proof}

\textbf{The freed person.}

One can notice that the index of a freed person is limited to a specific set of numbers. For example, for $UD$ dealing, the freed person always has an odd index. For $DU$ dealing, starting with 2 cards, the index is always even. For $UUD$, the index is never a multiple of 3. In general, we have the following.

\begin{proposition}\label{lem:perodic-freed-person-modularity-constraints}
    For any pattern $P$ of period $p$ and for any $N \geq u_1$, we have that for any index $i \leq p$ corresponding to a $D$ in $P$, we must have $F^P(N) \not\equiv i \pmod N$.
\end{proposition}
\begin{proof}
    If $F^{P}(N) \equiv i \pmod p$, then the card at index $F^{P}(N)$ is dealt in the first round of dealing. Since $P$ contains at least one $U$ in the first $N$ moves, this means there will be at least one card remaining to be dealt at this point, contradicting the assumption that $F^P(N)$ is the freed card.
\end{proof}

\textbf{The order of elimination of the first person.} One simple observation we can make about the elimination order of the first person is as follows.
\begin{proposition}
    For any pattern $P$ of period $p$ beginning with a $U$, if $N = kp + d_i - 1$ for some $k$ and some $i \leq |P_p|_D$, then the elimination order of the first person is $k|P_p|_D + i$.
\end{proposition}
\begin{proof}
    In the first round of dealing, we skip over the first person. We then complete $k$ full periods of $P$, following which we deal $i-1$ additional cards, as the $N$th card corresponds to index $d_i-1$ in the period. Then, after the dealing round has concluded, we are left to perform a ``down'' move and deal the first card. So, exactly $k|P_p|_D + i - 1$ cards are eliminated before the first card is.
\end{proof}

Another observation also applies to patterns starting with $U$, where the deck size is a multiple of the period.
\begin{proposition}\label{prop:periodic-eliminiationfirstperson-divisible}
    For any $N$, and any pattern $P$ of period $p$ beginning with a $U$, we have
    \[E^{P}(N p) = E^{P}(N |P_p|_U) + N|P_p|_D.\]
\end{proposition}

\begin{proof}
    Since $P$ begins with a $U$, in the first round of dealing, we skip the first card. 
    Since the number of cards in the deck is a multiple of the period, after this round of dealing, we are left to perform $P$ dealing again, now on a deck of $N |P_p|_U$ cards, beginning with card number $1$. So, the card numbered $1$ will be the $E^{P}(N |P_p|_U)$th card dealt after this point --- since we have already dealt $N|P_p|_D$ cards in our $N$ rounds of dealing, the statement follows.
\end{proof}

\subsection{Dealing every \texorpdfstring{$x$}{x}th card}
\label{sec:xthcard}

A particularly natural periodic pattern, generalizing the example of under-down dealing, is to deal every $x$th card for some $x$. That is, we deal with the repeating pattern $U^{x-1}D$.

\textbf{Number of moves.} The number of moves, in this case, is $xN$. 

\textbf{The dealing triangle.} Similar to Proposition~\ref{prop:trianglebuilding}, we have the following general statement.

\begin{proposition}\label{prop:skip-x-trianglebuilding}
For any $k \geq x$,
\[ T_{N,k}= \begin{cases}
1, &\text{ if } k=x; \\
T_{N-1,k-x}+1, & \text{ if } k \neq x.
\end{cases}\]
\end{proposition}

\begin{proof}
This follows from Theorem~\ref{thm:periodic-dealing-triangle-later-cols}. Here, the period $p$ is equal to $x$, and $|P_p|_D = |U^{x-1}|_D = 1$, so that theorem states $T_{N, k} = T_{N - 1, k-x} + 1$.
\end{proof}

\begin{example}
For $x = 3$, the fourth row is $4, 2, 1, 3$. Prepending zero gives us $0, 4, 2, 1, 3$. We shift the last $x - 1$, i.e., last $2$, elements to the front to get $1, 3, 0, 4, 2$. Adding $1$ to everything gives us $2, 4, 1, 5, 3$. This is the $5$th row for the triangle with $x = 3$. 
\end{example}

We noticed slanted diagonals before when we studied triangle $T$, where moving two steps to the right and one down increases the value by 1. In general, we have $T_{n,k} = T_{n-1, k-x-1} + 1$ whenever $k > x-1$.

We can provide a nice description of how to build row $xk$ of the triangle $T^P$ from row $xk-k$. Take row $xk-k$ and add $k$ to every term. Then insert numbers 1, 2, $\ldots$, $k$ in order while skipping $x-1$ terms. For example, when $k=3$, if row $2k = \{ a_1, a_2, a_3, \dots , a_{2k}\}$, then row $3k$ is $\{a_1+k, a_2+k, 1, a_3+k, a_4+k, 2, \dots, a_{2k}+k, k\}$.

\textbf{The Josephus triangle.} 
We can write the following as a special case of Theorem~\ref{thm:josephus-periodic}:
\begin{proposition} We have
    \[J^{U^{x-1}D}_{N,k} = \begin{cases}
    2k, & \text{if $k \leq x/2$ };\\
    J^{U^{x-1}D}_{N-1, k - 1} + x, & \text{if $k > x/2$ and $J^{U^{x-1}D}_{N-1, k - 1} \leq N-x$};\\
    \ceil{\paren{\frac{x}{x-1}}\paren{x - N + J^{U^{x-1}D}_{N-1, k - 1}}} - 1, & \text{otherwise}.
\end{cases}\]
\end{proposition}
\begin{proof}
      Pattern $U^{x-1}D$ has period $x$, with $|P_x|_D = 1$ and $d_k = xk$, $u_k = \ceil{\frac{xk}{x-1}} - 1$ for all $k$. Plugging into Theorem~\ref{thm:josephus-periodic} gives the statement.
\end{proof}

As a corollary of Proposition~\ref{prop:stabilization}, the Josephus triangle stabilizes with the $k$th column being $xk$.

\textbf{The freed person.} Instead of eliminating every second person, in this version of the Josephus problem, we eliminate every $x$th person. The sequences $F^P(N)$ are available in the database for $x$ ranging from 2 to 6 inclusive. They are, in order,
\[A006257,\ A054995,\ A088333,\ A181281,\ A360268.\]

We can write the following recursion.

\begin{proposition}
\label{prop:dealeveryx-freedperson}
We have
\[F^{U^{x-1}D}(N)=\begin{cases}
1, & \text{if } N = 1; \\
\left(F^{U^{x-1}D}(N-1)+x-1\pmod N\right)+1, & \text{ if } N >1.
\end{cases}\]  
\end{proposition}

\begin{proof}
    After dealing one card, we are left to perform $U^{x-1}D$ dealing on the cards
    \[x \pmod N + 1,\ \dots,\ N,\ 1,\ \dots,\ x \pmod N.\]
    The $i$th card in that list has number $\left(i + x - 1 \pmod N\right)+1$, and the freed person will be the $F^{U^{x-1}D}(N-1)$th card in that list.
\end{proof}

\textbf{The order of elimination of the first person.} An easy statement about $E^P$ in this case is: $E^P(xj-1) = j$. In particular $E^P(x^j-1) = x^{j-1}$. We can also specialize Proposition~\ref{prop:periodic-eliminiationfirstperson-divisible} to this case:

\begin{proposition}
    We have
    \[E^{U^{x-1}D}(N x) = E^{U^{x-1}D}(N (x-1)) + N.\]
\end{proposition}
\begin{proof}
    Pattern $U^{x-1}D$ has period $x$, with $|P_x|_D = 1$ and $|P_{x}|_U = x-1$. Plugging into Proposition~\ref{prop:periodic-eliminiationfirstperson-divisible} gives the statement.
\end{proof}

Consider $x=3$, which is the smallest case we have not studied yet. We denote this pattern as UUD, and we call such dealing \textit{under-under-down}. We study this dealing and other dealings of period three in the following section.

\subsection{Dealings of period three}
\label{sec:period3}

We consider patterns $P$ that are periodic with period 3. Note that $UUU$ is not defined as we want cards to be dealt at some point, and that $DDD = D$, which we describe below.

\begin{example}
\label{ex:D}
If $P = D$, the cards are always dealt in order --- so, $T_{n, k} = k$, and $J_{N,k} = k$. From this, we get that the elimination order of the first person is always $1$, and the freed person is always $n$.
\end{example}

\textbf{The number of moves.}
The number of moves for each of these patterns is shown in Table~\ref{tab:NoFMoversPeriod3}.

\begin{table}[ht!]
\centering
\renewcommand{\arraystretch}{1.5}
\begin{tabular}{|c|c|c|c|c|c|c|}
\hline
Pattern & $UUD$ & $UDU$ & $DUU$ & $UDD$ & $DUD$ & $DDU$ \\
\hline
Number of moves & $3N$ & $3N-1$ & $3N-2$ & $\ceil{\frac{3N}{2}}$ & $\floor{\frac{3N}{2}}$ & $\ceil{\frac{3N}{2}}-1$ \\
\hline
\end{tabular}
\caption{Number of moves for patterns of period 3.}
\label{tab:NoFMoversPeriod3}
\end{table}

\textbf{The dealing triangles.} The dealings of period $3$ involve some nice slanted diagonals, letting us express the whole triangle in terms of the first three columns.

For the patterns with only one $D$, we have the following.

\begin{proposition}
\label{prop:uud-dealing-triangle}
    If $P$ is any of $DUU$, $UDU$, or $UUD$, for any $k > 3$ we have 
    \[T^P_{N,k} = T^P_{N-1, k-3} + 1.\]
    We also start with $T^P_{1,1} = 1$. For the first 3 columns, we have
\begin{align*}
 T^{DUU}_{N,1} &= 1 & T^{DUU}_{N,2} &= T^{DUU}_{N-1, N-2} + 1 & T^{DUU}_{N,3} &= T^{DUU}_{N-1, N-1}+1 \\ 
 T^{UDU}_{N,1} &= T^{UDU}_{N-1, N-2} + 1 & T^{UDU}_{N,2} &= 1 & T^{UDU}_{N,3} &= T^{UDU}_{N-1, N-1} + 1 \\ 
 T^{UUD}_{N,1} &= T^{UUD}_{N-1, N-2} + 1 & T^{UUD}_{N,2} &= T^{UUD}_{N-1, N-1} + 1 & T^{UUD}_{N,3} &= 1.
\end{align*}
\end{proposition}
\begin{proof}
    These follow from Theorems \ref{thm:periodic-dealing-triangle-later-cols} and \ref{thm:periodic-dealing-triangle-first-cols}.
\end{proof}

For patterns with two $D$s, we have the following.
\begin{proposition}
    If $P$ is any of $DDU$, $DUD$, or $UDD$, for any $k > 3$ we have 
    \[T^P_{N,k} = T^P_{N-2, k-3} + 2.\]
    We also start with $T^P_{1,1} = 1$. For the first 3 columns, we have 
\begin{align*}
 T^{UDD}_{N, 1} &= T^{UDD}_{N-2, N-2}+2 & T^{UDD}_{N,2} &= 1 & T^{UDD}_{N,3} &= 2\\
 T^{DUD}_{N,1} &= 1 & T^{DUD}_{N,2} &= T^{DUD}_{N-2, N-2}+2 & T^{DUD}_{N,3} &= 2 \\ 
 T^{DDU}_{N,1} &= 1 & T^{DDU}_{N,2} &= 2 & T^{DDU}_{N,3} &= T^{DDU}_{N-2, N-2}+2.
\end{align*}
\end{proposition}
\begin{proof}
    These follow from Theorems \ref{thm:periodic-dealing-triangle-later-cols} and \ref{thm:periodic-dealing-triangle-first-cols}.
\end{proof}

Observe that these formulas are shifts of each other --- this can be derived from our formulas for appending $U$ and $D$. Table~\ref{tab:period3Dealingtriangles} shows the dealing triangles for length-$3$ patterns.

\begin{center}
\begin{table}[h]
\begin{tabular}{ccc}
\begin{subtable}{0.3\textwidth}
\centering
\begin{tabular}{cccccc}
1 &   &   &   &   &   \\
1 & 2 &   &   &   &   \\
1 & 2 & 3 &   &   &   \\
1 & 3 & 4 & 2 &   &   \\
1 & 5 & 3 & 2 & 4 &   \\
1 & 3 & 5 & 2 & 6 & 4 
\end{tabular}
\caption{$DUU$}
\end{subtable} & 
\begin{subtable}{0.3\textwidth}
\centering
\begin{tabular}{cccccc}
1 &   &   &   &   &   \\
2 & 1 &   &   &   &   \\
3 & 1 & 2 &   &   &   \\
2 & 1 & 3 & 4 &   &   \\
4 & 1 & 5 & 3 & 2 &   \\
4 & 1 & 3 & 5 & 2 & 6 
\end{tabular}
\caption{$UDU$}
\end{subtable} & 
\begin{subtable}{0.3\textwidth}
\centering
\begin{tabular}{cccccc}
1 &   &   &   &   &   \\
1 & 2 &   &   &   &   \\
2 & 3 & 1 &   &   &   \\
4 & 2 & 1 & 3 &   &   \\
2 & 4 & 1 & 5 & 3 &   \\
6 & 4 & 1 & 3 & 5 & 2 
\end{tabular}
\caption{$UUD$}
\end{subtable} \\
\begin{subtable}{0.3\textwidth}
\centering
\begin{tabular}{cccccc}
1 &   &   &   &   &   \\
1 & 2 &   &   &   &   \\
1 & 2 & 3 &   &   &   \\
1 & 2 & 4 & 3 &   &   \\
1 & 2 & 5 & 3 & 4 &   \\
1 & 2 & 5 & 3 & 4 & 6 
\end{tabular}
\caption{$DDU$}
\end{subtable} & 
\begin{subtable}{0.3\textwidth}
\centering
\begin{tabular}{cccccc}
1 &   &   &   &   &   \\
1 & 2 &   &   &   &   \\
1 & 3 & 2 &   &   &   \\
1 & 4 & 2 & 3 &   &   \\
1 & 4 & 2 & 3 & 5 &   \\
1 & 5 & 2 & 3 & 6 & 4 
\end{tabular}
\caption{$DUD$}
\end{subtable} & 
\begin{subtable}{0.3\textwidth}
\centering
\begin{tabular}{cccccc}
1 &   &   &   &   &   \\
2 & 1 &   &   &   &   \\
3 & 1 & 2 &   &   &   \\
3 & 1 & 2 & 4 &   &   \\
4 & 1 & 2 & 5 & 3 &   \\
6 & 1 & 2 & 5 & 3 & 4 
\end{tabular}
\caption{$UDD$}
\end{subtable}
\end{tabular}
\caption{Dealing triangle $T^P$ for patterns of period $3$}
\label{tab:period3Dealingtriangles}
\end{table}
\end{center}

\textbf{The Josephus triangles.} Table~\ref{tab:period3Dealingtriangles-inv} shows the Josephus triangles for length-3 patterns.

\begin{center}
\begin{table}[H]
\centering
\begin{tabular}{ccc}
\begin{subtable}{0.3\textwidth}
\centering
\begin{tabular}{cccccc}
1 &   &   &   &   &   \\
1 & 2 &   &   &   &   \\
1 & 2 & 3 &   &   &   \\
1 & 4 & 2 & 3 &   &   \\
1 & 4 & 3 & 5 & 2 &   \\
1 & 4 & 2 & 6 & 3 & 5
\end{tabular}
\caption{$DUU$}
\end{subtable} & 
\begin{subtable}{0.3\textwidth}
\centering
\begin{tabular}{cccccc}
1 &   &   &   &   &   \\
2 & 1 &   &   &   &   \\
2 & 3 & 1 &   &   &   \\
2 & 1 & 3 & 4 &   &   \\
2 & 5 & 4 & 1 & 3 &   \\
2 & 5 & 3 & 1 & 4 & 6 
\end{tabular}
\caption{$UDU$}
\end{subtable} & 
\begin{subtable}{0.3\textwidth}
\centering
\begin{tabular}{cccccc}
1 &   &   &   &   &   \\
1 & 2 &   &   &   &   \\
3 & 1 & 2 &   &   &   \\
3 & 2 & 4 & 1 &   &   \\
3 & 1 & 5 & 2 & 4 &   \\
3 & 6 & 4 & 2 & 5 & 1 
\end{tabular}
\caption{$UUD$}
\end{subtable} \\[1em]
\begin{subtable}{0.3\textwidth}
\centering
\begin{tabular}{cccccc}
1 &   &   &   &   &   \\
1 & 2 &   &   &   &   \\
1 & 2 & 3 &   &   &   \\
1 & 2 & 4 & 3 &   &   \\
1 & 2 & 4 & 5 & 3 &   \\
1 & 2 & 4 & 5 & 3 & 6 
\end{tabular}
\caption{$DDU$}
\end{subtable} & 
\begin{subtable}{0.3\textwidth}
\centering
\begin{tabular}{cccccc}
1 &   &   &   &   &   \\
1 & 2 &   &   &   &   \\
1 & 3 & 2 &   &   &   \\
1 & 3 & 4 & 2 &   &   \\
1 & 3 & 4 & 2 & 5 &   \\
1 & 3 & 4 & 6 & 2 & 5 
\end{tabular}
\caption{$DUD$}
\end{subtable} & 
\begin{subtable}{0.3\textwidth}
\centering
\begin{tabular}{cccccc}
1 &   &   &   &   &   \\
2 & 1 &   &   &   &   \\
2 & 3 & 1 &   &   &   \\
2 & 3 & 1 & 4 &   &   \\
2 & 3 & 5 & 1 & 4 &   \\
2 & 3 & 5 & 6 & 4 & 1 
\end{tabular}
\caption{$UDD$}
\end{subtable}
\end{tabular}
\caption{Josephus triangle $J^P$ for patterns of period $3$}
\label{tab:period3Dealingtriangles-inv}
\end{table}
\end{center}

\subsection{\texorpdfstring{$UD$}{UD} example}

We now delve deeper into our main example: $P = UD$. We start with the elimination order of the first person.

\textbf{Elimination order of the first person.} Consider the case of $E^{UD}$, which is $T^{UD}_{N,1}$. The corresponding sequence, A225381, entry has a recursive definition of this sequence: $T_{N,1} = \frac{N+1}{2}$ for odd $N$ and $T_{N,1} = T_{\frac{N}{2},1} + \frac{N}{2}$ for even $N$.

We suggest another description of this sequence in terms of the binary representation of $N$. Recall that the $2$-adic valuation of an integer $n$, denoted $\nu_2(n)$, is the exponent of the highest power of 2 that divides $n$.

\begin{proposition}
\label{prop:firstcolumn}
We can express the term $T_{N,1}$ as
\[T_{N,1} = N - \left\lfloor \frac{N}{2^{\nu _{2}(N)+1}} \right\rfloor = N - \frac{N}{2^{\nu _{2}(N)+1}} + \frac{1}{2}.\]
\end{proposition}

\begin{proof}
We want to show that 
\[T_{2n,1}=T_{n,1}+n.\]
First,  from Proposition~\ref{prop:trianglebuilding}, we have $T_{2n,1}=T_{2n-1,2n-1}+1$. We now use the fact from the same Proposition~\ref{prop:trianglebuilding} that $T_{N,k}=T_{N-1,k-2}+1$; applying it $n-1$ times we get that $T_{2n-1,2n-1}=T_{n,1}+n-1$. Thus, $T_{2n,1}=T_{n,1}+n$.

Suppose $N = 2^am$, where $m$ is odd. Then, using the property above $a$ times, we get
\[T_{2^am,1}=T_{m,1}+ 2^{a-1}m + 2^{a-2}m + \cdots + m = \frac{m+1}{2} + 2^am- m = N - \frac{m}{2} + \frac{1}{2}.\]
What is left to notice is that
\[\frac{m}{2} = \frac{N}{2^{\nu _{2}(N)+1}},\]
concluding the proof.
\end{proof}

In other words, to calculate the order of elimination of the first person, we express $N$ in binary, remove all the trailing zeros and the last 1, and subtract the resulting number from $N$.

\begin{example}
Consider $n=14$, which in binary is 1110. The trailing zeros with the last one form string 10; after removing it, we get string 11, which corresponds to 3 in binary. Subtracting 3 from 14, we get $T_{14,1}  = 11$. As another example, the reader can check that $T_{2^n,1} = 2^n$.
\end{example}

The following property of the first column is easy to notice and derive from Proposition~\ref{prop:firstcolumn}. This property plays an important role in the magic trick ``The Love Ritual''.

\begin{corollary}
Consider a subsequence of numbers divisible by $2^m$ and not $2^{m+1}$: $2^m$, $3\cdot 2^m$, $5\cdot 2^m$, and do on. The values of the first column with such indices form an arithmetic progression starting with $2^m$ and the difference $2^{m+1}-1$.
\end{corollary}

\begin{proof}
By Proposition~\ref{prop:firstcolumn}, we have $T_{2^m,1} = 2^m$. The difference between consecutive terms in the subsequence is
\[T_{(2j+1)2^m,1} - T_{(2j-1)2^m,1} = (2j+1)2^m - (2j-1)2^m - \frac{(2j+1)2^m}{2^{m+1}} + \frac{(2j-1)2^m}{2^{m+1}} = 2^{m+1} - 1.\qedhere\]
\end{proof}

\textbf{The dealing triangle.} As we mentioned before, the formula for the first column allows us to write the formula for the whole dealing triangle.

\begin{theorem}
We have
\[T_{N,2j+1} = N - \frac{N-j}{2^{\nu_2(N-j)+1}} + \frac{1}{2}  \quad \textrm{ and } \quad T_{N,2j} =j.\]
\end{theorem}
\begin{proof}
Corollary~\ref{cor:tablereduction} gives us the second part $T_{N,2j} =j$; it also gives us
\[T_{N,2k+1} = T_{N-k,1} + k.\]
Now using Proposition~\ref{prop:firstcolumn} we get
\[ T_{N,2k+1} = k + (N-k) - \frac{N-k}{2^{\nu_2(N-k)+1}} + \frac{1}{2} = N - \frac{N-k}{2^{\nu_2(N-k)+1}} + \frac{1}{2}.\qedhere\]
\end{proof}

In other words, the formula for $T_{N,k}$ if $k$ is odd using binary representation is the following. Take the binary representation of $N-\frac{k-1}2$ chopping of the trailing 0s and the last 1, then subtract the resulting number from $N$.

\textbf{Slanted anti-diagonals.} We also have some slanted anti-diagonals, where the same value is repeated along a diagonal moving $1$ up and $2$ to the right each step. That is, we have the following:
\begin{theorem}\label{thm:anti-diagonals}
    For the dealing triangle $T=T^{UD}$, and odd $k > 2$, if either
    \begin{itemize}
        \item $N$ is odd and $k \equiv 1 \pmod 4$, or
        \item $N$ is even and $k \equiv 3 \pmod 4$, 
    \end{itemize}
    then we have $T_{N, k}= T_{N+1, k-2}$.
\end{theorem}

\begin{proof}
    These anti-diagonals come from shifts of the even-numbered 
    columns. If we have an even number of cards, then after one dealing round ($N/2$ dealt cards), we are left to perform $UD$ dealing on cards $1,3,5,\dots, N-1$. Then, in the next dealing round, we will deal in order all cards with value $3 \pmod 4$. So, if $k \equiv 3 \pmod 4$, we will have $T_{N,k} = \frac{N}{2} + \left\lceil \frac{k}{4} \right\rceil$. 

    On the other hand, if we have an odd number of cards, then after one dealing round (consisting of $(N-1)/2$ dealt cards), we are left to perform $DU$ dealing on cards $1,3,5,\dots, N$. This next dealing round will deal in order all cards with value $1 \pmod 4$, so if $k\equiv 1 \pmod 4$, we will have $T_{N,k} = \frac{N-1}{2} + \left\lceil \frac{k}{4} \right\rceil$.

    This means that, for $k \equiv 3 \pmod 4$ and $N$ even, we have
    \[T_{N, k} = \frac{N}{2} +  \left\lceil \frac{k}{4} \right\rceil =  \frac{N}{2} + \left\lceil \frac{k-2}{4} \right\rceil = T_{N + 1, k-2},\]
    and for $N$ odd and $k \equiv 1 \pmod 4$, we have
    \[T_{N, k} = \floor{\frac{N}{2}} + \left\lceil \frac{k}{4} \right\rceil = \floor{\frac{N}{2}} + 1 + \left\lceil \frac{k-2}{4} \right\rceil = T_{N+1, k-2}.\qedhere\]
\end{proof}

For $UD$ dealing, this creates nice anti-diagonals in the triangle where a value is repeated.

Note that the statement is not true if we remove the parity constraints. For example, for $(n,k)=(9,7)$, we have $T_{9,7}=8$ but $T_{10,5}=10$.

\section{Special deck sizes}
\label{sec:specialdecksizes}

Given a pattern $P$, one might additionally be interested in which values of $N$ result in a particular index being freed. The simplest example of this would be considering which values of $N$ result in the first person surviving. For any pattern $P$, we can define a corresponding sequence $S^P$ consisting of all values of $N$ such that $F^P(N) = 1$ --- or, equivalently, such that $E^P(N) = N$.
We note that, if $P$ begins with a $D$, the sequence $S^P$ has only a single term. Indeed, the first person will be the first to be eliminated.

We could also consider the sequence of deck sizes for which the \emph{last} person is freed --- in other words, where $F^P(N) = N$. We denote this sequence $L^P$. This is interesting even for patterns beginning with $D$. The sequence numbers are listed in Table~\ref{tab:SLsequences}.

\begin{proposition}
    Integer $N>1$ belongs to $S^P$ if and only if $|P_N|_U$ belongs to $S^{P'}$, where $P'$ is $P$ with the first $N$ letters removed and $P$ starts with $U$.

    Integer $N>1$ belongs to $L^P$ if and only if either $P_N$ contains no $U$s, or $|P_N|_U$ belongs to $L^{P'}$, where $P'$ is $P$ with the first $N$ letters removed and $P_N = U$.
\end{proposition}

\begin{proof}
    After a single round of dealing, we are left to perform $P'$ dealing on a deck of $P_N|_U$ cards. So, we deal the first card last if and only if we skip it on the first round of dealing (meaning that $P_1 = U$) and then deal it last among the remaining cards (meaning that $|P_N|_U$ belongs to $S^{P'}$). Similarly, we deal the last card last if and only if we skip it on the first round of dealing (meaning that $P_N = U$) and then deal it last among the remaining cards (making that $|P_N|_U$ belongs to $L^{P'}$) --- with the exception that the last card is also dealt last if $P_N$ consists only of $D$s. 
\end{proof}

For $S^{UDD}$, we have the following recursive formula.
\begin{proposition}\label{prop:UDD-first-person-survives}
    We have
    \[S^{UDD}(2m + 1) = 3^m \quad \textrm{ and } \quad S^{UDD}(2m) = 2 \cdot 3^{m-1}.\]
\end{proposition}
\begin{proof}
    Suppose the deck size is $N = 3^x y$ for some $y$ not divisible by $3$ such that $F^{UDD}(N) = 1$. After $x$ rounds of dealing, we are left to perform $UDD$ dealing on $y$ cards, as each round decreases the deck size by a factor of 3. Since $y$ is not divisible by $3$, the leftover dealing pattern after one round begins with $D$, and if $y>2$, there will be at least two cards left after this round --- so the first card will not be dealt last. If $y = 1$, then the first card is the only remaining card and so is dealt last. If $y = 2$, then the penultimate round of dealing will deal the second card before the first, so the first card is dealt last.
\end{proof}

Now, comparing the initial terms, we get the following result.
\begin{corollary}
    Sequence $S^{UDD}$ is sequence A038754.
\end{corollary}

We note a simple relationship between the $L$ and $S$ sequences.

\begin{proposition}\label{prop:freelast-prependU}
    For any pattern $P$, we have $L^P(m) = S^{UP}(m)$.
\end{proposition}
\begin{proof}
    After performing a single $U$ move, we are left to perform $P$ dealing on cards numbered $2, \dots, N, 1$. This will free the card numbered $1$ if and only if $P$ dealing on $N$ cards frees the last card.    
\end{proof}

We also point out relationships for $L$ sequences.

\begin{proposition}\label{prop:freelast-prependD}
    For any pattern $P$, we have $L^P(1) = 1$ and $L^{DP}(m + 1) = L^{P}(m) + 1$.
\end{proposition}
\begin{proof}
    In a deck of size $1$, the last card is freed as it is the only card.
    For a deck of $N+1$ cards, if we perform $DP$ dealing, we start by eliminating the first card, and then perform $P$ dealing on the remaining $N$ cards, so the final card is freed if and only if the final card among $N$ cards is freed in $P$ dealing.
\end{proof}

Using the above Propositions, we can give formulas for several periodic patterns.

\begin{proposition}
    We have 
\[L^{DU}(m) = 2^{m-1} \quad \textrm{ and } \quad  L^{UD}(m) = 2^m - 1.\]
\end{proposition}
\begin{proof}
    We know that $S^{UD}(m) = 2^{m-1}$, so $L^{DU}(m) = 2^{m-1}$ follows from Proposition~\ref{prop:freelast-prependU}. Then, $L^{UD}(m) = 2^m - 1$ follows from Proposition~\ref{prop:freelast-prependD}.
\end{proof}

For $UDD$, $DUD$, and $DDU$, we can use Proposition~\ref{prop:UDD-first-person-survives} to derive the following.

\begin{proposition}
    We have
    \[L^{DDU}(2m + 1) = 3^m \quad \textrm{ and } \quad L^{DDU}(2m) = 2 \cdot 3^{m-1};\]
    \[L^{DUD}(2m + 1) = 2 \cdot 3^{m} - 1 \quad \textrm{ and } \quad L^{DUD}(2m) = 3^m - 1;\]
    \[L^{UDD}(2m + 1) = 3^m - 2  \quad \textrm{ and } \quad L^{UDD}(2m) = 2 \cdot 3^{m} - 2.\]
\end{proposition}
\begin{proof}
    The first bullet point follows from Proposition~\ref{prop:freelast-prependU} and Proposition~\ref{prop:UDD-first-person-survives}. Then, we can derive the other two from Proposition~\ref{prop:freelast-prependD}.
\end{proof}

The recursions above allow us to confirm the sequence numbers.
\begin{corollary}
    Sequence $L^{UDD}$ is A164123, sequence $L^{DUD}$ is A062318, and sequence $L^{DDU}$ is A038754.
\end{corollary}

For $UUD$, we have the following observation.
\begin{proposition}
    If $L^{UUD}(m)$ is odd, then the next term is given by 
    \[L^{UUD}(m+1) = \frac{3L^{UUD}(m) + 1}{2}.\]
\end{proposition}
\begin{proof}
    Let $N = L^{UUD}(m)$.
    We claim that $T_{N + i, 3i - 1} = N + i$ for all $i = 1, \dots, N$. We prove this by induction.
    By Proposition~\ref{prop:uud-dealing-triangle}, we know that $T^{UUD}_{N+1 , 2} = T_{N, N} + 1 = N+1$.
    Now, assuming the statement for some $i \geq 1$, Proposition~\ref{prop:uud-dealing-triangle} guarantees that $T^{UUD}_{N+(i+1), 3(i+1) - 1} = T^{UUD}_{N + i, 3i - 1} + 1 = N + i + 1$, so we have proven the claim inductively. To deduce the proposition, observe that for any $i < \frac{3N + 1}{2}$, we have $3i - 1 < N + i$, so the last person is eliminated before the end --- however, if $N$ is odd, then for $i = \frac{3N + 1}{2}$ then $3i - 1 = N + i$, so the last person is the freed person. 
\end{proof}

\section{Magic}
\label{sec:magic}

\subsection{Overview of the tricks }

In this section, we discuss magic tricks related to the different dealing patterns discussed above. We start with tricks related to knowing the freed person. In all of these tricks, we allow the audience to choose a card. Then we put the card into the deck and manipulate the deck in various ways. In the final step, we have the deck face down and deal using the pattern $P$. We throw away the dealt card and reveal the last card to be the one chosen by the audience.

Here is the list of tricks and mathematical facts behind them.

\begin{itemize}
    \item \textbf{Know your freed person.} The magician knows $F^P(N)$ and can place the target card at the corresponding spot.
    \item \textbf{Spelling bee.} The target card is the top card before the $P$ dealing. In this particular trick $F^{SUD}(26) = 1$. 
    \item \textbf{Double-dealing.} The target card is the bottom card before the $P$ dealing: $F^P(N) = N$.
    \item \textbf{Ace quartet.} The target card is positioned at a special place, then the deck size is reduced and manipulated so that the target card is on top and the smaller deck size belongs to $S^P$. In this particular trick, we use the fact that $F^{UUUD}(5) = 1$. We suggest a generalization, where we use a known value of $F^P(N)$.
\end{itemize}

One of the most famous existing tricks using under-down dealing is called `The Love Ritual'. It allows the audience to manipulate the deck in many seemingly random ways. All the manipulations are designed in such a way that the target card stays at the bottom. Before the last dealing, the audience can discard 0 to 3 cards. This is the most magical place in the trick. Independent of the number of discarded cards, the subsequent $UD$ dealing reveals the target card as the last. We generalize this trick to other dealing patterns.

\begin{itemize}
    \item \textbf{The generalized love ritual.} This trick uses the fact that the freed person sequence consists of subsequences in arithmetic progression.
\end{itemize}

We also propose other tricks related to other mathematical ideas.

\begin{itemize}
    \item \textbf{The power of powers of two.} We use that fact that $J^{UD}_{2^n, 2^{n-1}} = 2^n$. We generalize this to other bases, using the corresponding fact that $J^{U^{b-1}D}_{b^a, b^{a-1}} = b^a$.
    \item \textbf{Stripes and stripes forever.} We use properties of the $UD$ dealing.
    \item \textbf{Second time's the charm.} The trick's generalization works when the permutation corresponding to row $N$ of the Josephus triangle has small order. In our particular example, we use the fact that row 6 of $J^{SUD}$ has order 2.
\end{itemize}

\subsection{Tricks relying on knowing who will be freed}

Several tricks rely on the performer knowing the value of the freed person for a particular dealing pattern and manipulating a chosen card to end up in that position. In this subsection, we describe several old and new tricks of this form.

\subsubsection{Know your freed person}

The original version of this simple card trick can be found online in the video 'Down Under Card Find | Self Working Aussie Magic' \cite{petty2024down}. It uses the under-down dealing. We describe a generalization for any pattern $P$.

\begin{enumerate}
\item Shuffle a deck of $N$ cards.
\item The audience chooses an index $i \in \{1, \dots, N\}$, and remembers the card at index $i$ from the top.
\item The performer claims to cut the deck randomly, but, in reality, moves the top $N - F^{P}(N)$ cards of the deck to the bottom.
\item The audience moves the top $i$ cards of the deck to the bottom. (The performer can claim the audience is ``using their own secret information to make the deck even more random''.)
\item The performer deals the cards according to the pattern $P$ until one card is left.
\item This last card is the original card that the audience thought of. 
\end{enumerate}

\begin{theorem}
    The `Know your freed person' trick works.
\end{theorem}
\begin{proof}
    After Steps 3 and 4, the total of $N + i - F^{P}(N)$ cards have been moved from the top of the deck to the bottom. Thus, the card originally at position $i$ now has a position congruent to $i - (N + i - F^{P}(N)) \pmod N$, meaning it is at position $F^{P}(N)$. So, it will be the last card dealt by $P$ dealing.
\end{proof}

\subsubsection{Spelling bee}

If the freed person is the first person, we note that the cutting Step 3 in the general trick above can be removed, leading to a simpler trick. We describe here an instantiation of such a card trick making use of the coincidental property that $26$ belongs to $S^{SUD}$. The magician can emphasize that they are using half of a standard deck because they are spelling, and half of the standard deck is the same size as the number of letters in the alphabet.

\begin{enumerate}
    \item A pile of 26 cards is prepared, one for each letter of the alphabet. For example, the magician can take all the red cards or any two suits. Or, the magician can say that to save time, they need a smaller deck and secretly take exactly half. The audience chooses one card and places it on top of the deck.
    \item The magician manipulates the deck, making sure that the top card stays on top, but trying to obfuscate this fact.
    \item SpellUnder-Down dealing is used to deal these cards into another pile. (The performer can lean into the spelling theme, claiming that we are now using the power of the alphabet to scramble the cards.)
    \item The final card is revealed to be the audience's chosen card.
\end{enumerate}

\begin{theorem}
The `Spelling bee' trick works.
\end{theorem}
\begin{proof}
    Since $F^{SUD}_{26} = 1$, the final card dealt when performing the SpellUnder-Down dealing on $26$ cards is equal to the card originally at the top.
\end{proof}

We can perform a similar trick for any $P$ and any $N \in S^P$. Many of the small values we found for various sequences $S^P$ in Section \ref{sec:specialdecksizes} could be suitable for other tricks with natural stories --- for instance, $SUD$ could also use a deck size of $50$ (for the states of the U.S.A.), or the magician can pretend that they have a standard deck secretly removing two cards.

Pattern $UD$ can use any power of 2; pattern $UDD$ can use 54 cards, which is the standard deck with two jokers; pattern $UUUD$ could use a deck size of $12$ (dealing every $4$th card for the $4$ seasons, in a deck of size $12$ for the $12$ months).

\subsubsection{Double-dealing}

We can similarly obtain a nice trick in the special case that the freed card is the last card. The following trick works for any dealing pattern $P$ and any $N$ such that $F^{P}(N) = N$, or, equivalently, $N \in L^P$. A good example would be $UUD$ with $N=13$ (the deck can consist of a single suit), or $UD$ with $N$ one less than a power of $2$.
\begin{enumerate}
    \item Given a pile of $N$ cards, the audience selects one card and remembers it. The magician puts the card at the bottom of the deck, maybe trying to hide this fact.
    \item The performer deals according to pattern $P$, placing the dealt cards in a new pile face down.
    \item The performer cycles the top card of the deck to the bottom, maybe trying to hide the fact.
    \item The performer allows the audience to, as many times as they want, take some of the top cards and place them in the middle of the deck, but not at the bottom. This step can be replaced or complemented by having the performer shuffle the deck, if they are able to do so in such a way that avoids changing the bottom card.
    \item The performer deals according to $P$ yet again. The dealt cards can be thrown out in a dramatic fashion. The last card dealt is revealed to be the chosen card.
\end{enumerate}

\begin{theorem}
The `Double-dealing' trick works.
\end{theorem}
\begin{proof}
    Since $F^{P}(N) = N$, after the 2nd step, the chosen card will be on top of the deck. Once it's cycled to the bottom, it will remain at the bottom despite the audience's shuffling. After dealing according to $P$ again, because $F^{P}(N) = N$, this will be the last card dealt.
\end{proof}

\subsubsection{Ace quartet}

Our trick here is a variant of a known trick called the ``king quartet''~\cite{fulves2001self}.

\begin{enumerate}
    \item The magician lets the audience select a card from a standard deck of cards. Then, the performer arranges four aces on top, followed by the chosen card. For example, the deck can be prearranged with 4 aces at the bottom, then the chosen card placed on top, then the cut made to move the bottom four aces to the top.
    \item The magician deals the top $5$ cards of the deck face down and discards the rest, thus moving the 5th card to the top.
    \item The magician then performs $UUUD$ dealing face up while spelling the letters A-C-E before each deal. This way, 4 aces are dealt.
    \item The final remaining card is revealed to be the audience's chosen card.
\end{enumerate}

\begin{theorem}
    The `Ace quartet' trick works.
\end{theorem}

\begin{proof}
    After Step 1, the top four cards are aces, followed by the chosen card. After Step 2, the top card is the chosen card, and the other four cards are aces. The rest of the proof is immediate from the fact that $F^{UUUD}(5) = 1$.
\end{proof}

The beauty of this method is that the initial deck size does not matter, as the deck size is adjusted in the middle to the size that works.

Similar ideas can be used when the target card is at a specific position. For example, if the target card is the fifth card, we can remove Step 2 in the above procedure and use the number of cards $N$, such that $F^P(N) = 5$.

For example, we know that $F^{UD}(2^n+2) = 5$. Therefore, we can perform the trick skipping Step 2 with 10 cards and $UD$ dealing when the target card is hidden in the fifth place.

\subsection{Generalized love ritual}

The most famous trick using dealing patterns is Aragon's ``Love Ritual'', a favorite of Penn and Teller.
The original trick starts with 4 cards ripped in half, then treating the 8 halves as a deck of 8 cards. The trick uses $UD$ dealing, and was generalized to an arbitrary number of cards by Park and Teixeira~\cite{teixeira2017mathematical}. Here, we present a new generalization of the trick to arbitrary dealing patterns $P$. The trick works as long as the pattern and deck size satisfy some particular conditions; we give several examples of cases that work.

The instructions to the audience are as follows:

\begin{enumerate}
    \item Take $N$ cards and shuffle them.
    \item Tear the cards in half and put one pile on top of the other.
    \item Take any number of cards from the top and put them on the bottom. This can be done many times.
    \item Take the top $N - 1$ cards and place them into the middle of the deck, anywhere except at the top or the bottom. These $N - 1$ cards may be shuffled in the process.
    \item Take the top card and put it in a safe place. At this stage, any instruction that rearranges the leftover cards without touching the bottom card can be added. For example, take the top card and place it in the middle.
    \item Throw away any number between $0$ and $k$ (inclusive) of cards from the top of the deck.
    \item Deal $\ell$ cards from the top of the deck to the bottom, where $\ell$ is a natural number satisfying $\ell \equiv - F^P(2N-1-i) \pmod {2N - 1 - i}$ for all $i \in \{0, \dots, k\}$.
    \item Perform $P$ dealing, throwing away each dealt card.
    \item Reveal the leftover card and the safe card; the cards should match.
\end{enumerate}

\begin{theorem}
    The `Generalized love ritual' works for any deck size $N$, and any parameter $k < 2N-1$, so long as we have 
\[F^P(i) \equiv F^P(j) \pmod{\gcd(i, j)}\]
for all $i,j \in \{2N - 1 - k,\dots,  2N - 1\}$.
\end{theorem}
\begin{proof}
    The first 5 steps are identical to the original love ritual. The cuts cycle the cards, meaning that after Step 3, the card at index $i$ matches the card at index $i+N$, for all $i \in \{1,2,\dots,N\}$. After Step 4, the top card matches the bottom card. After Step 5, the hidden card will still match the bottom card, and the deck will have $2N-1$ cards.
    
    After Step 6, we will have removed some number $x \in \{0,\dots, k\}$ from the top of the deck, leaving behind a size-$(2N - 1 - x)$ deck with the card of interest at the bottom. Then, in Step 7, we move $\ell$ cards from the top to the bottom, shifting the card of interest's index from $2N -1 - x$ to $(2N - 1 - x - \ell) \pmod {2N - 1 - x}$, which by our choice of $\ell$ corresponds to $F^P(2N - 1 - x)$. So, the last dealt card will match the safe card.

    In order to prove that there exists a choice of $\ell$ satisfying $\ell \equiv - F^P(2N-1-i) \pmod {2N - 1 - i}$ for all $i \in \{0, \dots, k\}$, we use the Chinese remainder theorem. The Chinese remainder theorem guarantees that there always exists a solution to a system of equations of the form $x = a_i \pmod{r_i}$ so long as $a_i \equiv a_j \pmod{\gcd(r_i, r_j)}$ for all $i,j$. Letting $a_i = -F^P(2N - 1 - i)$ and $r_i = 2N - 1 - i$, this condition is exactly what we have assumed in the theorem statement, so we know that a solution $\ell$ exists.
\end{proof}

In particular, we note that this trick can be performed with $U^xD$ dealing, where we can take $\ell = x (2N-1)$.
\begin{corollary}
    The `Generalized love ritual' works for any $P = U^xD$, any $N$ such that $F^{U^xD}(2N - 1) = 2N - 1$, and any $k \leq \floor{\frac{2N - 2}{x+1}}$.
\end{corollary}
\begin{proof}
Observe by Proposition~\ref{prop:dealeveryx-freedperson} that the values of $F^{P}(2N-1-i)$ for $i = 0, \dots, \floor{\frac{2N-2}{x+1}}$ are $2N-1$, $2N -1 - (x+1)$, $\dots$, $2N - 1 - \floor{\frac{2N-2}{x+1}}(x+1)$. These values form an arithmetic progression, and so satisfy the necessary conditions for the Chinese remainder theorem --- we can take $\ell = x (2N-1)$. 
\end{proof}

\begin{example}
    In the original love ritual, the dealing pattern $P$ is $UD$, and the deck size $N$ is 4. Our formula suggests to take $\ell = x(2N-1) = 1 \cdot (2 \cdot 4 - 1) = 7$. The freed person sequence starts as 1, 1, 3, 1, 3, 5, 7. We see that the last four terms form an arithmetic progression, implying that conditions are satisfied if $k \leq 3$, matching our formula.
\end{example}

We also note that, for any pattern, this trick works for $k \leq 2$.

\begin{corollary}
    The `Generalized love ritual' works for any $P$, any $N>1$, and any $k \leq 2$.
\end{corollary}
\begin{proof}
    Since $2N - 1$ is odd, we must have that all of $2N-1$, $2N - 2$, and $2N - 3$ are pairwise coprime. Thus, the Chinese remainder theorem guarantees some valid choice of $\ell$.
\end{proof}

\subsection{The power of powers of two}

A trick shown in the online video `The Down Under card trick tutorial!' \cite{carddealingking2023} follows the following sequence of instructions using a deck of 16 cards. We describe the trick when the number of cards in the deck is any power of two.
\begin{enumerate}
    \item Take $2^a$ cards, shuffle them, and show the bottom card to everyone to remember.
    \item $UD$-deal all cards face down until you run out of cards.
    \item Turn the top card of the dealt pile face up.
    \item Cut the cards. This can be done multiple times. 
    \item Deal the cards into two piles by alternating piles while dealing. Remove the pile without the card face up. 
    \item Repeat the previous step $a-1$ times until a two-card pile is left, one of the cards face-up.
    \item The card not faced up is the original bottom card remembered by the audience.    
\end{enumerate}

\begin{theorem}
    `The power of powers of two' trick works.
\end{theorem}
\begin{proof}
   Number the cards as 1 through $2^a$ top to bottom after they have been shuffled in Step 1. The card $2^a$ was shown to the audience, and we call it the target card.
   
   Step 2 is similar to the elimination in the Josephus problem. The first card from the bottom is the first card eliminated. Thus, looking from the bottom up, the cards in the resulting pile are in order that match row $2^a$ of the Josephus triangle.
   
   We know that $J^{UD}_{2^a, x} = 2x$ for all $x \leq 2^{a-1}$, implying that $J^{UD}_{2^a, 2^{a-1}} = 2^a$. This means that the position of the card $2^a$ --- the target card --- is number $2^{a-1}+1$ from the top. That means the card is $2^{a-1}$ away from the top card, which we turn face up.
   
  Step 4 cycles the cards, and because the face-up card is $2^{a-1}$ cards away from the card the audience remembers, cycling the cards does not change the number of cards between them. 
   Then, as we repeat Step 5, observe by induction that we will always have the face-up card and the target card exactly $2^{i-1}$ cards apart, where $2^i$ is the current size of the deck. If this is true for a given $i$, their positions have the same remainder modulo $2^{i-1}$, and thus they will end up in the same pile together --- and the number of cards in between will be halved, so they will now be exactly $2^{i-2}$ cards apart, finishing the induction.
   
   When we get to the end, the only two remaining cards will, therefore, be the face-up card and the target card.
\end{proof}

We can also generalize this trick to powers with other bases. Fix any natural numbers $b$ and $n$, and any pattern $P$ such that $J_{b^n, b^{n-1}}^P= b^{n}$. For instance, we can take the pattern $P = U^{b-1}D$. The trick is as follows.
\begin{enumerate}
    \item The performer takes $b^n$ cards, shuffles them, and shows the bottom card to the audience.
    \item The performer performs $P$ dealing on the deck, producing a new face-down deck.
    \item The performer flips over the top card of this new deck.
    \item The audience cuts the deck. This can be done multiple times.
    \item While there are more than $b$ cards remaining, the performer deals the cards one at a time into $b$ piles, and eliminates all piles except the one containing the flipped card.
    \item After this, there will be left a pile of $b$ cards. If the flipped card is at the top of the pile, the performer reveals that the bottom card of the pile was the chosen card. Otherwise, the performer reveals that the card right on top of the flipped card was the chosen card.
\end{enumerate}

\begin{theorem}
    `The power of powers of $b$' trick works.
\end{theorem}
\begin{proof}
    Since $J_{b^n, b^{n-1}}^P= b^{n}$, after Step 2, the chosen card is now at position $b^{n-1}$ from the bottom, meaning that the chosen card is $b^n - b^{n-1}$ cards later than the flipped card. This will remain true after cutting the deck. Then, at each step of dealing, since the indices of the flipped card and the chosen card are the same mod $b$, they will be placed in the same pile, and the distance between them will be divided by $b$. At the end, they will be in the same pile, with the chosen card $b - 1$ cards later than the flipped card --- if the flipped card is at the top, this means that the chosen card is at the bottom of the pile, and otherwise it means that the chosen card is directly above the flipped card.
\end{proof}

\subsection{Stripes and stripes forever}

We suggest another simple card trick making use of under-down dealing.

\begin{enumerate}
\item Take $2N$ cards, alternating between red and black.
\item Allow the audience to cut the deck as many times as they want.
\item Perform under-down dealing face down to re-order the deck.
\item Split the deck into two equal halves (top and bottom), allow the audience to shuffle both halves, and use the perfect riffle shuffle to recombine them.
\item Allow the audience again to cut the deck as many times as they want.
\item The performer reveals that, despite all these modifications, the deck remains alternating between red and black.
\end{enumerate}

\begin{theorem}
    The `stripes and stripes forever' trick works.
\end{theorem}

\begin{proof}
    Cutting the deck maintains the property of alternating colors. Performing under-down dealing will deal every even-index card first, so the final dealt pile has 26 cards of one color followed by the cards of the other color. So, when the deck is split and recombined, we will again have alternating colors.
\end{proof}

\subsection{\texorpdfstring{$k$}{k}th time's the charm}

This trick makes use of multiple rounds of the same dealing. Fix any pattern $P$, and let $k$ be the order of the permutation which is the $N$ row of $J^P$ --- i.e., $k$ is the smallest integer such that applying the permutation $k$ times yields the identity. A nice example of practical performance is the SpellUnder-Down dealing, with $N=6$ and $k=2$. The second row of the Josephus triangle is 4, 2, 5, 1, 3, 6, and the order of this permutation is 2. More examples of dealing patterns yielding Josephus permutations of small order can be found in the GitHub repository ``kthtimesthecharm''~\cite{boya2025}.

Unfortunately, for periodic dealings, the phenomena above cannot occur beyond the small decks. For example, the rows $N$, where $N > 3$, in $J^{UD}$ start as 2, 4, and, therefore, they cannot correspond to a permutation of order 2.

The trick proceeds thus.

\begin{enumerate}
    \item The magician takes a deck of $N$ cards, and shows the audience that they are in order. The performer then repeats the following steps $k-1$ times:
    \begin{enumerate}
        \item Deal cards according to $P$, placing each dealt card face up in a pile. The audience will be able to see that the cards are dealt out of order.
        \item Flip the deck over so that it is face down.
    \end{enumerate}
    \item At this point, the audience is given the chance to deal according to $P$. (The performer can claim to have given up on their hopes of getting the trick to work, and sulkily challenge the audience to see if any of them can do better.) The cards are now dealt in order.
\end{enumerate}

\begin{theorem}
The `$k$th time's the charm' trick works.
\end{theorem}
\begin{proof}
    When Step 1a is performed, cards are dealt in the order $J^P_N$ --- but since they are placed on top of each other, they eventually end up in the reverse order of the Josephus permutation. So, after Step 1b, they end up in the same order as row $N$ of the Josephus triangle. Since this permutation has order $k$, none of the first $k-1$ iterations of the process will put the cards in order, but the $k$th will.
\end{proof}

\begin{example}
     The magician puts the deck of 6 cards face down on the table. Secretly from the audience, the cards are arranged in order from 1 to 6, where 1 is on top. The magician announces, ``In order to be a magician, you need to experiment with the card deck'', and gives the steps of `experimenting'.
     \begin{enumerate}
         \item Put the top card on the bottom of the deck.
        \item Put the top 2 cards on the bottom of the deck.
        \item Put the top 3 cards on the bottom of the deck.
     \end{enumerate}

This ends up creating the same deck from the start, but it sets a start to the story. The magician then continues with the SpellUnder-Down dealing, by saying, ``But magic also needs practice. We must first practice the trick. It may not work the first time''. The dealt cards seem to be in a random order, $(4,2,5,1,3,6)$, confusing the audience. Then the magician flips the deck over to make it face down, and says, ``So we try again, as with more practice, the card trick may work''. The magician performs the SpellUnder-Down dealing. But this time, the cards get revealed as the audience says the number: ``O-N-E'' and then the 1 card is dealt. ``T-W-O'' and the 2 card is dealt, and so on. The magician ends by congratulating the audience on becoming magicians. 
\end{example}

\section{Acknowledgments}

We are grateful to the PRIMES STEP program for allowing us to conduct this research.

------------------------------------------

Existing sequences:

A000012, A000079, A000096, A000225, A001651, A005408, A005843, A006257, A007494, A008585, A008586, A016777, A016789, A032766, A038754, A054995, A062318, A081614, A081615, A088333, A152423, A164123, A181281, A182459, A225381, A291317, A321298, A337191, A360268, A378982.

New sequences:

A378635, A378674, A378682, A380195, A380201, A380202, A380204, A380246, A380247, A380248, A381048, A381049, A381050, A381051, A381114, A381127, A381128, A381129, A381151, A381591, A381622, A381623, A381667, A382354, A382355, A382356, A382358, A382528, A383076, A383845, A383846, A383847, A384770, A384772, A384774, A385327, A385328, A385333, A385513, A386639, A386641, A386643, A386305, A386312.

\end{document}